\documentclass[11pt]{article}
\usepackage{epic,latexsym,amssymb}
\usepackage{color}
\usepackage{tikz}
\usepackage[colorinlistoftodos]{todonotes}
\usepackage{pgfplots}

\usepackage{mathrsfs}
\usepackage[colorlinks]{hyperref}
\usepackage{amsfonts,epsf,amsmath}
\usepackage{cite}
\usepackage{lineno}
\usepackage[lined,ruled,linesnumbered]{algorithm2e}
\usepackage{multicol}
\usepackage{subfig}
\usepackage{epic}
\usepackage{amssymb}
\usepackage{latexsym}
\usepackage{tikz}
\usetikzlibrary{decorations.markings}
\usepackage{pgf}

\usetikzlibrary{arrows}

\usepackage{comment}

\newtheorem{theorem}{Theorem}

\newtheorem{lemma}{Lemma}

\newtheorem{definition}{Definition}
\newtheorem{claim}{Claim}
\newtheorem{subclaim}{Claim}[claim]

\newtheorem{observation}{Observation}

\newtheorem{problem}{Problem}

\newcommand{\CG}{{\rm CG}}

\newcommand{\SP}{{\rm{SP}-\rm{graph} }}
\newcommand{\SPs}{\rm{SP}-\rm{graphs} }
\newcommand{\SCC}{\rm{SCC}}
\newcommand{\SC}{\rm{SC }}

\newcommand{\qed}{$\Box$}
\newcommand{\QED}{$\Box$}

\newcommand{\smallqed}{{\tiny ($\Box$)}}

\newcommand{\1}{\vspace{0.1cm}}
\newcommand{\2}{\vspace{0.2cm}}

\let\oldenumerate\enumerate
\renewcommand{\enumerate}{
  \oldenumerate
  \setlength{\itemsep}{1pt}
  \setlength{\parskip}{0pt}
  \setlength{\parsep}{0pt}
}

\def\vertex(#1){\put(#1){\circle*{2}}}
\def\vertexo(#1){\put(#1){\circle{2}}}
\def\vert(#1){\put(#1){\circle*{1.5}}}
\def\verto(#1){\put(#1){\circle{1.5}}}
\def\lab(#1)#2{\put(#1){\makebox(0,0)[c]{#2}}}
\begin{document}

\title{Singleton Coalition Graph Chains }

\author{$^1$Davood Bakhshesh, \, $^2$Michael A. Henning\thanks{Research supported in part by the University of Johannesburg.}, \, and $^3$Dinabandhu Pradhan\thanks{Corresponding author.} \\ \\
$^1$Department of Computer Science\\
University of Bojnord \\
Bojnord, Iran \\
\small \tt Email: d.bakhshesh@ub.ac.ir \\
\\
$^2$Department of Mathematics and Applied Mathematics \\
University of Johannesburg \\
Auckland Park, 2006 South Africa\\
\small \tt Email: mahenning@uj.ac.za\\
\\
$^3$Department of Mathematics and Computing \\
Indian Institute of Technology (ISM) \\
Dhanbad, India \\
\small \tt Email: dina@iitism.ac.in}

\date{}
\maketitle

\begin{abstract}
Let $G$ be graph  with vertex set $V$ and order $n=|V|$. A coalition in $G$  is a combination of two distinct sets, $A\subseteq V$ and $B\subseteq V$, which are disjoint and are not dominating sets of $G$, but their union $A\cup B$ is a dominating set of $G$. A coalition partition of $G$ is a partition $\mathcal{P}=\{S_1,\ldots, S_k\}$ of its vertex set $V$, where each set $S_i\in \mathcal{P}$ is either a dominating set of $G$ with only one vertex, or it is not a dominating set but forms a coalition with some other set $S_j \in \mathcal{P}$. The coalition number ${C}(G)$ is the maximum cardinality of a coalition partition of $G$. To represent a coalition partition $\mathcal{P}$ of $G$, a coalition graph $\CG(G, \mathcal{P})$ is created, where each vertex of the graph corresponds to a member of $\mathcal{P}$ and two vertices are adjacent if and only if their corresponding sets form a coalition in $G$. A coalition partition $\mathcal{P}$ of $G$ is a singleton coalition partition if every set in $\mathcal{P}$ consists of a single vertex. If a graph $G$ has a singleton coalition partition, then $G$ is referred to as a singleton-partition graph. A graph $H$ is called  a singleton coalition graph of a graph $G$ if there exists a singleton coalition partition $\mathcal{P}$ of $G$ such that the coalition graph $\CG(G,\mathcal{P})$ is isomorphic to $H$. A singleton coalition graph chain  with an initial graph $G_1$ is defined as the sequence $G_1\rightarrow G_2\rightarrow G_3\rightarrow \cdots$ where all graphs $G_i$ are singleton-partition graphs, and $\CG(G_i, \Gamma_1)=G_{i+1}$, where $\Gamma_1$ represents a singleton coalition partition of $G_i$. In this paper, we address two open problems posed by Haynes et al. We characterize all graphs $G$ of order $n$ and minimum degree $\delta(G)=2$ such that $ C(G )= n$. Additionally, we investigate the singleton coalition graph chain starting with graphs $G$ where $\delta(G)\le 2$.
\end{abstract}


\section{Introduction}

Let $G = (V(G),E(G))$ be a graph with vertex set $V(G)$ and edge set $E(G)$, and of order $n(G) = |V(G)|$ and size $m(G) = |E(G)|$. If the graph $G$ is clear from context, we write $V$ and $E$ rather than $V(G)$ and $E(G)$. The \emph{open neighborhood} $N_G(v)$ of a vertex $v$ in $G$ is the set of vertices adjacent to $v$, while the \emph{closed neighborhood} of $v$ is the set $N_G[v] = \{v\} \cup N_G(v)$. For a set $S \subseteq V(G)$, its \emph{open neighborhood} is the set $N_G(S) = \cup_{v \in S} N_G(v)$, and its \emph{closed neighborhood} is the set $N_G[S] = N_G(S) \cup S$.  We denote the \emph{degree} of $v$ in $G$ by $\deg_G(v) = |N_G(v)|$. The vertex $v$ is called a \emph{full vertex} if it is adjacent to all vertices $V \setminus \{v\}$, that is, if $N_G[v] = V$.  An \emph{isolated vertex} is a vertex of degree~$0$, and an \emph{isolate}-\emph{free graph} is a graph that contains no isolated vertex. The minimum and maximum degrees in $G$ are denoted by $\delta(G)$ and $\Delta(G)$, respectively. For a positive integer $k$, we let $[k] = \{1, \ldots, k\}$.

A subset $D$ of $V$ is called a \emph{dominating set} in $G$ if every vertex of $G$ not in $D$ is adjacent to at least one vertex of $D$. The minimum cardinality of a dominating set of $G$ is the \emph{domination number} of $G$, denoted by $\gamma(G)$. Domination in graphs is now very well studied in the literature. If $X,Y \subseteq V$, then set $X$ dominates the set $Y$ if $Y \subseteq N_G[X]$, that is, every vertex $y \in Y$ belongs to $X$ or is adjacent to a vertex of $X$. A comprehensive treatment of domination in graphs can be found in the recent so-called ``domination books''~\cite{HaHeHe-20,HaHeHe-21,HaHeHe-23,HeYe-book}.

In~2020, Haynes, Hedetniemi, Hedetniemi, McRae, and Mohan~\cite{coal0} introduced and studied the concept of a \emph{coalition} in a graph. They defined a pair of sets $A,B \subseteq V$ to be a \emph{coalition} in $G = (V,E)$ if neither $A$ not $B$ is a dominating sets in $G$, but $A \cup B$ is a dominating set of $G$. Such a pair $A$ and $B$ is said to \emph{form a coalition}, and $A$ and $B$ are called \emph{coalition partners}.

A vertex partition ${\cal P} = \{S_1,S_2,\ldots, S_k\}$ of $V$ is a \emph{coalition partition} of $G$, abbreviated a $c$-\emph{partition} in~\cite{coal0}, if every set $S_i\in {\cal P}$ is either a dominating set of $G$ with cardinality $|S_i|=1$ or is not a dominating set of $G$ but forms a coalition with some other set $S_j$ in the partition ${\cal P}$. The maximum cardinality of a coalition partition of $G$ is the \emph{coalition number} of $G$,  denoted by $\mathcal{C}(G)$. A motivation of this graph theoretic model of a coalition is given by Haynes et al. in their series of papers on coalitions in~\cite{coal0,coal1,coal2,coal3}. In these papers, a special emphasis is on studying the coalition number in trees and cycles. Additionally, in \cite{coal2}, Haynes et al. introduced upper bounds on the coalition number of a graph in terms of the minimum and maximum degree.

In \cite{coal0}, Haynes et al. defined a coalition graph as follows. Given a coalition partition ${\cal P} = \{S_1, \ldots, S_k\}$ of a graph $G$, a graph called the \emph{coalition graph} of $G$, denoted by $\CG(G,{\cal P})$, is assigned to this partition, where the vertices correspond to the members of $\cal P$ and two vertices are adjacent if and only if the corresponding members in ${\cal P}$ form a coalition in the graph $G$.  In \cite{coal1}, Haynes et al. examined coalition graphs in trees, paths, and cycles. Additionally, in \cite{coal3} they showed that every graph is the coalition graph of some graph. In \cite{coal0}, Haynes et al. posed the following open problem.

\begin{problem}[\cite{coal0}]
\label{prob1}
{\rm Characterize all graphs $G$ of order $n$ with  $C(G)=n$.}
\end{problem}

In \cite{bakhcoal}, Bakhshesh et al.  characterized  all graphs $G$ of order $n$ with $\delta(G) \le 1$ whose coalition number is $n$, thereby partially solving Problem~\ref{prob1}.  They also identified all trees whose coalition number is $n-1$. In this paper, we make further progress in solving  Problem~\ref{prob1}. For this purpose, we characterize all graphs $G$ of order $n$ and $\delta(G)=2$ satisfying $C(G)=n$.

In \cite{coal3}, Haynes et al. defined the following three concepts associated with a coalition graph:  \emph{singleton-partition graph}, \emph{singleton coalition graph} and \emph{singleton coalition graph chain}. Let $G$ be a graph of order $n$. Let $\Gamma_1$  be the partition of $V(G)$ such that every member of $\Gamma_1$ is a singleton set (a set with only one element). We call such a coalition partition of $G$ a \emph{$\Gamma_1$-partition} of $G$. The graph $G$ is a \emph{singleton-partition graph} ({\emph{SP}-\emph{graph}) if it has a $\Gamma_1$-partition. If $G$ is a \SP, then the coalition graph $\CG(G,\Gamma_1)$ is called the \emph{singleton coalition graph} ({\emph{SC}-\emph{graph}) of $G$. A \emph{singleton coalition graph chain} with an initial graph $G_1$ is defined as a sequence
\[
{\cal S}_{G_1} := G_1\rightarrow G_2 \rightarrow G_3 \rightarrow \cdots \rightarrow G_k
\]
where all graphs $G_i$ are singleton-partition graphs, and $\CG(G_i, \Gamma_1)=G_{i+1}$ for all $i \in [k-1]$. The length of the sequence ${\cal S}_{G_1}$ is defined as $k-1$. If ${\cal S}_{G_1}$ is an infinite sequence of graphs $G_i$, then its length is considered to be infinity ($\infty$). By convention, if all graphs $G_i$ in ${\cal S}_{G_1}$ are isomorphic, the length of ${\cal S}_{G_1}$ is defined to be zero.  We denote a sequence ${\cal S}_{G_1}$ of maximum possible length (starting from the initial graph $G_1$) by \emph{$G_1$}-\emph{\SC chain} and we denote its length by $L_{\SCC}(G_1)$. For example, the $C_4$-SC chain and $C_5$-SC chain are given by
\[
C_4\rightarrow K_4\rightarrow \overline{K}_4
\hspace*{0.5cm} \mbox{and} \hspace*{0.5cm}
C_5\rightarrow C_5\rightarrow C_5\rightarrow \cdots ,
\]
respectively, while the $P_3$-SC chain is given by the sequence
\[
P_3\rightarrow K_1\cup K_2 \rightarrow P_3 \rightarrow K_1 \cup K_2 \rightarrow \cdots.
\]

Hence, $L_{\SCC}(C_4)=2$,  $L_{\SCC}(C_5)=0$, and $L_{\SCC}(P_3)=\infty$. In~\cite{coal3},  Haynes et al. posed the following open problem.

\begin{problem}[\cite{coal3}]
\label{prob3}
{\rm Investigate singleton coalition graph chains.}
\end{problem}

In addition to presenting Problem \ref{prob3}, Haynes et al. also posed the following question: ``Do arbitrarily long singleton coalition graph chains exist?''   In this paper, we show that there exist graphs $G$ such that the length of the $G$-SC chain is infinity. Moreover, we characterize $G$-SC chains for all singleton coalition graph $G$ with $\delta(G)\le 2$.

\section{Preliminaries}

In this section, we present some known and preliminary results on the coalition number of graphs, as well as additional definitions. In \cite{bakhcoal}, the authors characterize all \SPs $G$ with $\delta(G) \in \{0,1\}$. For the graphs with $\delta(G)=0$, they proved the following result.

\begin{theorem}[\cite{bakhcoal}]
\label{thmdelta0}
If $G$ is a graph of order~$n$ with $\delta(G) = 0$, then ${C}(G) = n$ if and only if $G \cong  K_1 \cup K_{n-1}$.
\end{theorem}

For graphs with $\delta(G)=1$ with exactly one full vertex, they proved the following result.

\begin{theorem}[\cite{bakhcoal}]
\label{lemdelta1full}
If $G$ is a graph of order~$n \ge 3$ with $\delta(G)=1$ and with exactly one full vertex, then ${C}(G)=n$ if and only if $G$ is obtained from the graph $K_1 \cup K_{n-1}$ by adding an edge joining the isolated vertex to an arbitrary vertex of the complete graph $K_{n-1}$.
\end{theorem}

For the graphs with $\delta(G)=1$ and with no full vertex, they defined a family $\mathcal{F}_1$ of graphs as follows.

\begin{definition}[\cite{bakhcoal}]{\rm (The family $\mathcal{F}_1$)}
\label{defn1}
{\rm Let $G$ be a graph constructed as follows. The vertex set of $G$ consists of $\{x, y, w\}$ and two disjoint sets, $P$ and $Q$, with $P \cap Q \cap \{x, y, w\} = \emptyset$ and $|P \cup Q| \ge 1$. If $Q$ is non-empty, then $|Q| \ge 2$. To define the edge set $E(G)$, we start by assigning $y$ as the unique neighbor of $x$, and so $N_G(x) = \{y\}$. We then join $w$ to all vertices in $P \cup Q$, and so $N_G(w) = P \cup Q$. For each vertex $p \in P$, we join it to all vertices in $(P \cup Q) \setminus \{p\}$. If $Q$ is non-empty, we join $y$ to all vertices in $Q$, and so $Q \cup \{x\}$ is a subset of $N(y)$. Furthermore if $Q \ne \emptyset$, then we add edges between vertices in $Q$, including the possibility of adding no edge, in such a way that $G[Q]$ does not contain a full vertex. Finally, we add any number of edges between $y$ and vertices in $P$. Let $G$ denote the resulting isolate-free graph. Let ${\cal F}_1$ be the family consisting of all such graphs $G$ constructed in this way. }
\end{definition}

We note that if $G \in \mathcal{F}_1$ has order~$n$ and is disconnected, then $G \cong K_2 \cup K_{n-2}$. In \cite{bakhcoal}, the authors  presented the following result.

\begin{theorem}[\cite{bakhcoal}]
\label{thmdelta1nf}
A graph $G$ with $\delta(G)=1$ and with no full vertex is a \SP if and only if $G\in{\cal F}_1$.
\end{theorem}

If $G$ is a \SP and $x$ and $y$ are distinct vertices in $G$ such that $\{x\}$ and $\{y\}$ form a coalition, then $\{x,y\}$ is a dominating set of $G$, and so $\{x,y\} \cap N[v] \ne \emptyset$ for every vertex $v \in V(G)$. We state this formally as follows.

\begin{observation}
\label{ob:ob1}
 If  $x$ and $y$ are distinct vertices and are not full vertices in a \SP $G$, then   $\{x\}$ and $\{y\}$ form a coalition if and only if  $\{x,y\} \cap N[v] \ne \emptyset$ for every vertex $v \in V(G)$.
\end{observation}

Now, we assume that $G$ is a \SP with $\delta(G)=1$ and with no full vertex.  By Theorem~\ref{thmdelta1nf}, $G\in {\cal F}_1$.  We next characterize the SC-graph $H$ of $G$. For this purpose, we  define a family ${\cal H}_1$ of graphs.

\begin{definition}{\rm (The family $\mathcal{H}_1$)}
\label{defnH1}
{\rm This family comprises of all bipartite graphs $H$ having two distinct parts $A_1=\{x_1,y_1\}$ and $B_1=P_1\cup \{w_1\}\cup Q_1$, where $w_1$ is a vertex, and  $P_1$ and $Q_1$ are sets of vertices that satisfy the condition $|P_1\cup Q_1|\ge 1$. Additionally, if $Q_1\ne \emptyset$, then $|Q_1|\ge 2$. The edges of $H$ are formed as follows: $y_1$ is adjacent to every vertex in $B_1$, while $x_1$ is adjacent only to the vertices of $P_1$ and $w_1$ (as shown in Figure \ref{figH}).}
\end{definition}

\begin{figure}[hbt]
\begin{center}
{\subfloat[]{\includegraphics[width = 0.4\textwidth]{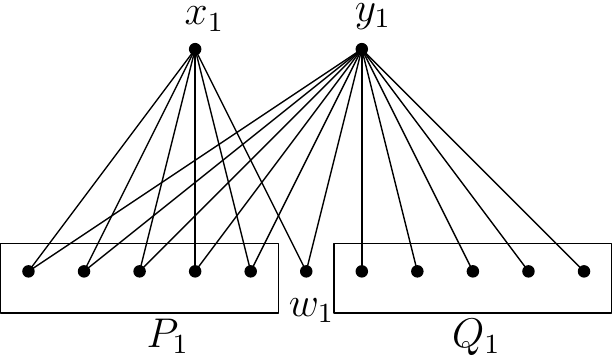}} }
\hspace{1cm}
{\subfloat[]{\includegraphics[width =0.22\textwidth]{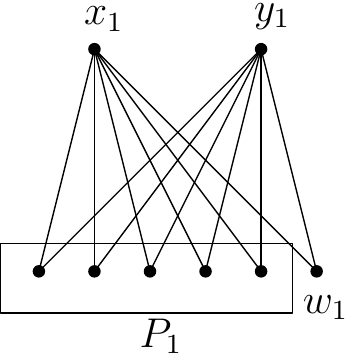}}}
\caption{(a): A graph of ${\cal H}_1$ with $Q_1\neq\emptyset$. (b): A graph of ${\cal H}_1$ with $Q_1= \emptyset$.}
\label{figH}
\end{center}
\end{figure}

The following notation is defined before proving the following theorem: Let $G$ be an \SP and $H$ the SC-graph of $G$. For a vertex $x$ in $G$, we represent the corresponding vertex in $H$ as~$\tilde{x}$. We are now in a position to prove the following result.

\begin{theorem}
\label{thmfh1cg}
If $G\in {\cal F}_1$, then $\CG(G,\Gamma_1)\in {\cal H}_1$.
\end{theorem}
\begin{proof}
Let $G\in {\cal F}_1$ with vertex set $V = V(G)$ and the edge set $E = E(G)$. Thus,  $V=\{x,y\}\cup P\cup Q \cup \{w\}$, where the vertices $x$, $y$ and $w$ and the vertex subsets $P$ and $Q$ satisfy the conditions defined for the graphs in Definition~\ref{defn1} that belong to the family ${\cal F}_1$. We note that $x$ is a vertex of $G$ with the minimum degree $\delta(G)=1$, and $N_G(x)=\{y\}$.  Let $H_1= \CG(G,\Gamma_1)$. By Observation~\ref{ob:ob1}, if $u \in P \cup \{w\}$, then $\{u\}$ and $\{x\}$ form a coalition, and so $\tilde{u}$  is adjacent to $\tilde{x}$. If $u \in P \cup Q \cup \{w\}$, then $\{u\}$ and $\{y\}$ form a coalition, and so $\tilde{u}$  is adjacent to $\tilde{y}$. Moreover, there is no coalition between $\{u\}$ and $\{u'\}$, where  $u,u' \in P\cup Q\cup \{w\}$, and so $\tilde{u}$ is not adjacent to $\tilde{u}'$ in $H_1$.  Since $G$ has no full vertices, $\{x\}$ and $\{y\}$ do not form a coalition, and so $\tilde{x}$ is not adjacent to $\tilde{y}$ in  $H_1$. If $q\in Q$, then since $G[Q]$ has no full vertex, $\{q\}$ and $\{x\}$ do not form a coalition, and so $\tilde{q}$ is not adjacent to $\tilde{x}$ in $H_1$. By these properties of the graph $H_1$, we infer that $H_1$ satisfies the conditions for the family ${\cal H}_1$, and therefore $H_1\in{\cal H}_1$.\qed
\end{proof}

\section{\SPs with minimum degree two}

In this section,  we characterize all graphs $G$ of order $n$ with $\delta(G)=2$ and $C(G)=n$.  We first present a family ${\cal F}_2$ of graphs.

\begin{definition}[Family ${\cal F}_2$]
\label{def-f2}
{\rm   Let $G$ be a graph constructed as follows. The vertex set of $G$ is defined as $V=\{x,y,z\}\cup L_1\cup R_1\cup R_2\cup L_2\cup W$, where $L_1, R_1, R_2,$ and $L_2$ are pairwise disjoint sets and where the set $W$ may intersect the set $L_1\cup R_1\cup R_2\cup L_2$. The vertex $x$ has degree~$2$ in $G$ with $y$ and $z$ as its unique neighbors, and so $N_G(x)=\{y,z\}$. We now add additional edges to $G$ as follows.}

\begin{enumerate}

\item {\textbf{The family} $\mathbf{{\cal F}_2^1}$.}
{\rm In this case, the sets $W$, $L_1$, $L_2$, and $R_2$ are defined to be empty and the set $R_1$ is non-empty. We join both vertices $y$ and $z$ to every vertex of $R_1$, but do not add an edge between $y$ and $z$. Further, we add edges between vertices in $R_1$, including the possibility of adding no edge. The resulting graph $G$ satisfies $\delta(G) = 2$ and has no full vertex.  Let ${\cal F}_2^1$ be the family consisting of all such graphs $G$ constructed in this way. An example of a graph in the family ${\cal F}_2^1$ is illustrated in Figure \ref{St1pic}.} \2

\item {\textbf{The family} $\mathbf{{\cal F}_2^2}$.}
{\rm  In this case, the sets $W$, $L_2$, and $R_2$ are defined to be empty and both sets $L_1$ and $R_1$ are non-empty. We join the vertex $y$ to every vertex of $L_1 \cup R_1$, but do not add an edge between $y$ and $z$. We join the vertex $z$ to every vertex of $R_1$ but to no vertex of $L_1$. Additionally, we add all edges between vertices in $L_1$ to form the clique $G[L_1]$. Further, we add edges between vertices in $R_1$, including the possibility of adding no edge. Finally, we add any additional edges to $G$ while maintaining a minimum degree of~$2$ (recall that the vertex $x$ has degree~$2$ in $G$) and not creating any full vertex. The resulting graph $G$ satisfies $\delta(G) = 2$ and has no full vertex.  Let ${\cal F}_2^2$ be the family consisting of all such graphs $G$ constructed in this way. An example of a graph in the family ${\cal F}_2^2$ is illustrated in Figure~\ref{St2pic}.} \2

\begin{figure}[ht]
\begin{center}
{\subfloat[]{\label{St1pic}\includegraphics[width = 0.3\textwidth]{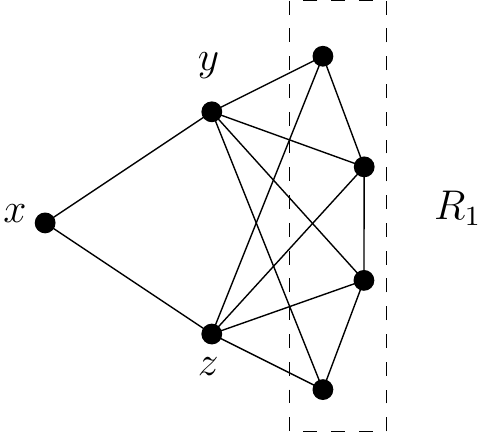}}
}
\hspace{1cm}
{\subfloat[]{\label{St2pic}\includegraphics[width =0.3\textwidth]{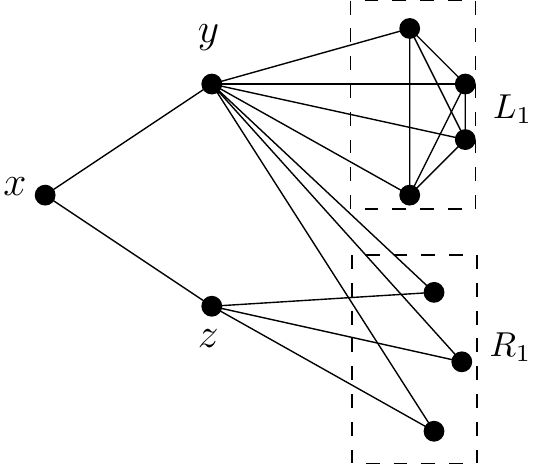}}
}
\caption{  (a): A graph in ${\cal F}_2^1$ (b): A graph in ${\cal F}_2^2$ }
\end{center}
\end{figure}

\item {\textbf{The family} $\mathbf{{\cal F}_2^3}$.}
{\rm  In this case, each of the sets $L_1$, $R_2$, and $W$ is non-empty. Further the set $L_2\subseteq W$, although possibly $L_2 = \emptyset$. We add all edges between vertices in $W$ to form the clique $G[W]$, and we join every vertex of $W$ to all vertices of $L_1 \cup R_1 \cup R_2$.  We join the vertex $y$ to every vertex of $L_1 \cup R_1$, but do not add an edge between $y$ and $R_2$. We join the vertex $z$ to every vertex of $R_1\cup R_2$, but do not add an edge between $z$ and $L_1$. If $R_1$ is not empty, then we join each vertex in $R_1$ to all vertices in either $L_1$ or $R_2$.

\hspace*{0.25cm} Now, either we do not add the edge $yz$, in which case we add all edges between vertices in $L_1$ to form the clique $G[L_1]$ and we add all edges between vertices in $R_2$ to form the clique $G[R_2]$, or we do add the edge $yz$, in which case we join each vertex in $L_1$ to all vertices in either $L_1$ or $R_2$, and do the same for each vertex in $R_2$.

\hspace*{0.25cm} Finally, we add any additional edges to $G$ while maintaining a minimum degree of~$2$ (recall that the vertex $x$ has degree~$2$ in $G$) and not creating any full vertex. The resulting graph $G$ satisfies $\delta(G) = 2$ and has no full vertex.  Let ${\cal F}_2^3$ be the family consisting of all such graphs $G$ constructed in this way. Examples of graphs in the family ${\cal F}_2^3$ are illustrated in Figure~\ref{figfamily}. } \2

\end{enumerate}
{\rm The family ${\cal F}_2$ is defined as ${\cal F}_2={\cal F}_2^1\cup {\cal F}_2^2\cup {\cal F}_2^3$. }
\end{definition}

\begin{figure}[ht]
\begin{center}
{\subfloat[]{\includegraphics[width = 0.45\textwidth]{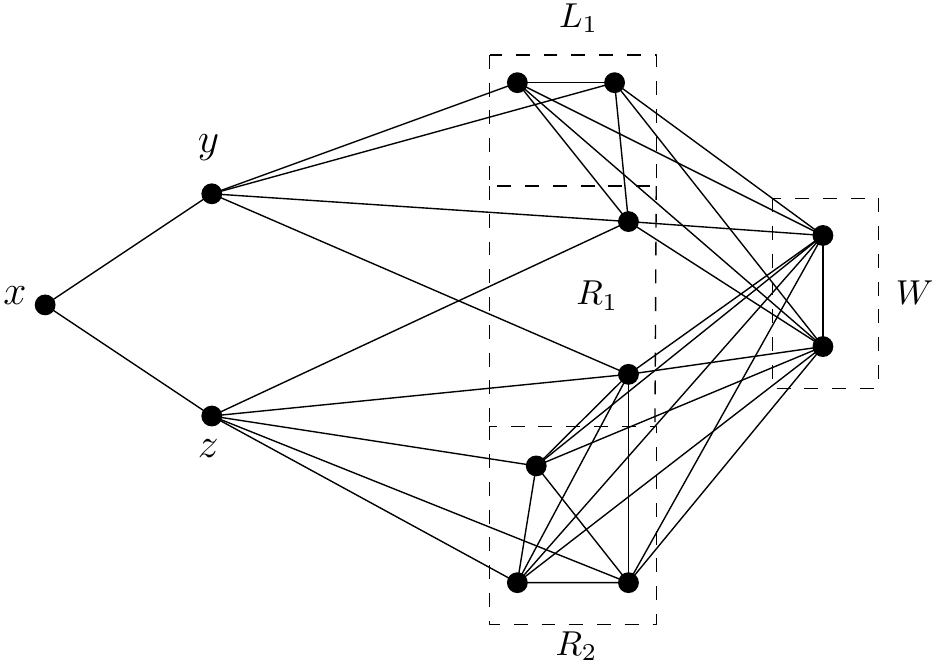}} }
{\subfloat[]{\includegraphics[width =0.45\textwidth]{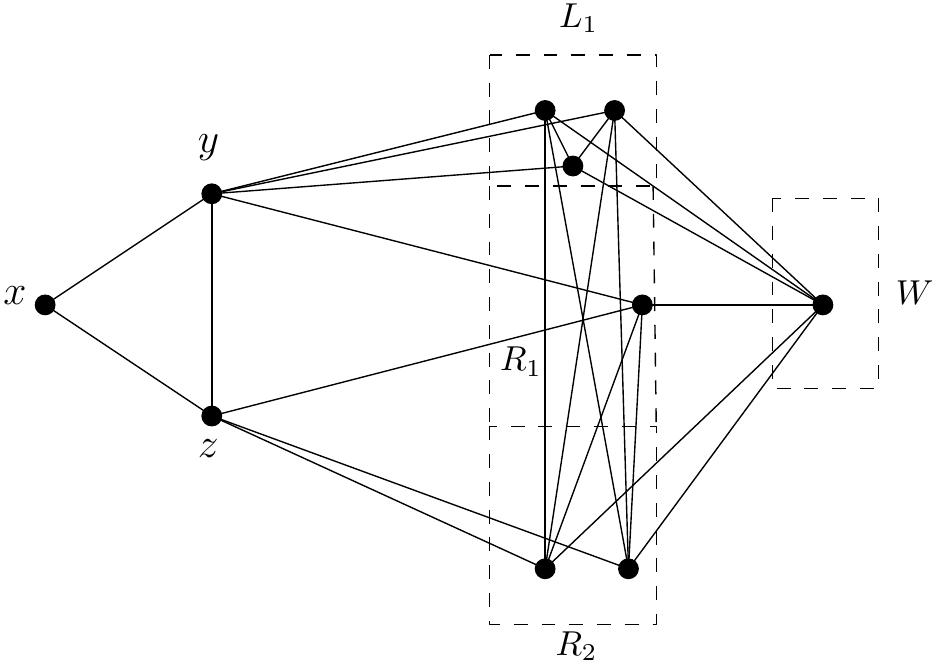}}}
\caption{Two different graphs in ${\cal F}_2^3$. (a): $(y,z)\not \in E$. (b): $(y,z) \in E$. } \label{figfamily}
\end{center}
\end{figure}

We determine next the SP-cycles with no full vertex. The coalition number of a cycle is determined in~\cite{coal0}. In particular, Haynes et al.~\cite{coal0} showed that if $G$ is a cycle $C_n$, then $C(C_n) = n$ if and only if $3 \le n \le 6$. Since $C_3$ has a full vertex, we consider here cycles $C_n$ for $4 \le n \le 6$. If $G$ is the $4$-cycle $xyazx$, then taking $R_1=\{a\}$, we infer that $G \in {\cal F}_2^1$. If $G$ is the $5$-cycle $xyabzx$, then taking $W =\{a,b\}$,  $L_1=\{a\}$,  $R_2=\{b\}$, and $R_1 = L_2=\emptyset$, we infer that $G \in {\cal F}_2^3$. If $G$ is the $6$-cycle $xyabczx$, then taking $W =\{b\}$,  $L_1=\{a\}$,  $R_2=\{c\}$, and $R_1 = L_2=\emptyset$, we infer that $G \in {\cal F}_2^3$. We state these observations formally as follows.

\begin{observation}
\label{ob:ob2}
For $n \ge 3$, if $C_n$ is a \SP, then $4 \le n \le 6$ and $G \in {\cal F}_2$.
\end{observation}

We show next that if $G$ is a \SP with $\delta(G)=2$ that contains no full vertex, then $G$ belongs to the family~${\cal F}_2$.

\begin{theorem}
\label{thmf2}
If $G$ is a \SP with $\delta(G)=2$ and with no full vertex, then $G \in {\cal F}_2$.
\end{theorem}
\begin{proof}
Let $G = (V,E)$ be a \SP with $\delta(G)=2$ and with no full vertex, and let $G$ have order~$n$, and so $C(G)=n$. Let $x$ be a vertex of degree~$2$ in $G$ with neighbors $y$ and $z$, and so $N_G[x] = \{x,y,z\}$. Further, let ${\cal P}$ be a singleton coalition partition of $G$. By Observation~\ref{ob:ob1}, all sets in ${\cal P} \backslash\{\{x\},\{y\},\{z\}\}$ form a coalition with at least one of $\{x\},\{y\}$, or $\{z\}$, while no two sets in ${\cal P}\backslash\{\{x\},\{y\},\{z\}\}$ form a coalition. In what follows, we let $V_x = V \setminus \{x,y,z \}$, and so $V = N_G[x] \cup V_x$. We proceed further with a series of claims.

\begin{claim}
\label{c:claim1}
If $\{x\}$ forms a coalition with both $\{y\}$ and $\{z\}$, then $G \in {\cal F}_2$.
\end{claim}
\begin{proof}
Suppose that $\{x\}$ forms a coalition with both $\{y\}$ and $\{z\}$. Thus, both $\{x,y\}$ and $\{x,z\}$ are dominating sets in $G$. Hence each of $y$ and $z$ is adjacent to all vertices in $V_x$. Since $G$ has no full vertex, it follows that $yz \notin E$. For each vertex $u\in V_x$, the sets $\{u,y\}$ and $\{u,z\}$ are dominating sets in $G$, and so $\{u\}$ forms a coalition with both $\{y\}$ and $\{z\}$. Thus, letting $R_1 = V \setminus \{x,y,z\}$ and $W =L_1=L_2=R_2=\emptyset$, we infer that $G\in {\cal F}_2^1$ (see Definition \ref{def-f2}). Therefore, $G \in {\cal F}_2$.~\smallqed
\end{proof}

\begin{claim}
\label{c:claim2}
If $\{x\}$ forms a coalition with exactly one of $\{y\}$ and $\{z\}$, then $G \in {\cal F}_2$.
\end{claim}
\begin{proof}
Suppose that $\{x\}$ forms a coalition with exactly one of $\{y\}$ and $\{z\}$. Renaming the vertices $y$ and $z$ if necessary, we may assume that $\{x\}$ forms a coalition $\{y\}$ but not with $\{z\}$. Thus, $\{x,y\}$ is a dominating sets in $G$, implying that vertex $y$ is adjacent to all vertices in $V_x$. Since $G$ has no full vertex, it follows that $yz \notin E$. Since $\{x,z\}$ is not a dominating set, there exists a vertex in $V_x$ that is not adjacent to $z$. Let $V_x = L_1\cup R_1$, where $L_1$ and $R_1$ are the sets of vertices in $V_x$ that are not adjacent to $z$ and are adjacent to $z$, respectively. Necessarily, $L_1$ is non-empty since $\{x,z\}$ is not a dominating set. Let $u \in V_x$. If $u \in R_1$, then since $uz \in E$ and all vertices in $V_x$ are adjacent to $y$, the sets $\{u\}$ and $\{y\}$ form a coalition. Now, suppose $u \in L_1$. Since $z$ is not adjacent to both $u$ and $y$, the sets $\{u\}$ and $\{y\}$ do not form a coalition. Hence, $\{u\}$ must form a coalition with either $\{z\}$ or $\{x\}$. If $\{u\}$ forms a coalition with $\{x\}$, then all vertices in $V_x$ are adjacent to $u$. If $\{u\}$ forms a coalition with $\{z\}$, then all vertices in $L_1$ are adjacent to $u$. In both cases, $G[L_1]$ is a clique. By letting $W=L_2=R_2=\emptyset$, we infer that $G \in {\cal F}_2^2$ (see Definition \ref{def-f2}).  Therefore, $G\in{\cal F}_2$.~\smallqed
\end{proof}

\begin{claim}
\label{c:claim3}
If $\{x\}$ forms a coalition with neither $\{y\}$ nor $\{z\}$, then $G \in {\cal F}_2$.
\end{claim}
\begin{proof}
Suppose that $\{x\}$ forms a coalition with neither $\{y\}$ nor $\{z\}$. Since $\cal P$ is a singleton coalition partition of $G$, there exists a vertex $w\in V_x$ such that $\{x\}$ and $\{w\}$ form a coalition. Let $W$ be the set of all vertices $w \in V_x$ such that $\{w\}$ and $\{x\}$ form a coalition. Thus, $\{x,w\}$ is a dominating set of $G$, implying that vertex $w$ is adjacent to all vertices in $V_x \setminus \{w\}$. Hence, $G[W]$ is a clique. Let $L_1$ be the set of all vertices in $V_x$ which are adjacent to $y$ and not are adjacent to~$z$. Further let $R_1$ be the set of all vertices in $V_x$ which are adjacent to both $y$ and $z$, and let $R_2$ be the set of all vertices in  $V_x$ which are adjacent to $z$ and not are adjacent to $y$. Let $L_2$ be the set of all vertices in $V_x$ which are adjacent to neither $y$ nor $z$ (see Figure \ref{ethmpic}). We note that $L_1,R_1,R_2$, and $L_2$ are pairwise disjoint. Since $W \subseteq V_x = L_1\cup R_1\cup R_2\cup L_2$, we note that $W \cap \left(L_1\cup R_1\cup R_2\cup L_2\right) \ne \emptyset$. Since $\{x,y\}$ is not a dominating set, $R_2\cup L_2\ne \emptyset$, and since $\{x,z\}$ is not a dominating set, $L_1\cup L_2 \ne \emptyset$. We proceed further with two subclaims.

\begin{figure}[htb]
\begin{center}
	\includegraphics[width=0.3\linewidth]{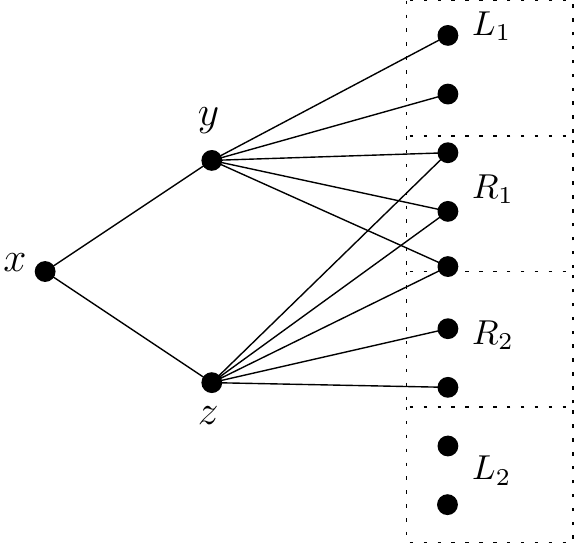}
	\caption{Illustrating the proof of Claim~\ref{c:claim3}.
}
  \label{ethmpic}
	\end{center}
\end{figure}

\begin{subclaim}
\label{c:claim3.1}
If $\{y\}$ and $\{z\}$ form a coalition, then $G \in {\cal F}_2$.
\end{subclaim}
\begin{proof}
Suppose that $\{y\}$ and $\{z\}$ form a coalition. Thus, $\{y,z\}$ is a dominating set, and so in this case we note that $L_2=\emptyset$.

Suppose that $yz \notin E$. We show firstly that $G[R_2]$ is a clique. Let $r_2 \in R_2$. Since $y$ is adjacent to neither $r_2$ nor $z$, the sets $\{r_2\}$  and $\{z\}$ do not form a coalition. Thus, $\{r_2\}$ forms a coalition with $\{y\}$ or $\{x\}$. If $\{r_2\}$ forms a coalition with $\{x\}$, then $r_2 \in W$. By our earlier observations, $G[W]$ is a clique, implying that $r_2$ is adjacent to all vertices in $R_2 \setminus \{r_2\}$. If $\{r_2\}$ forms a coalition with $\{y\}$, then $\{r_2,y\}$ is a dominating set of $G$, once again implying that $r_2$ is adjacent to all vertices in $R_2 \setminus \{r_2\}$. Hence, $G[R_2]$ is a clique. We show secondly that every vertex in $R_1$ is adjacent to all vertices of $L_1$ or $R_2$ (or both $L_1$ and $R_2$). Let $r_1 \in R_1$. If $\{r_1\}$ forms a coalition with $\{x\}$, then $r_1 \in W$. If $\{r_1\}$ forms a coalition with $\{y\}$, then every vertex of $R_2$ must be adjacent to $r_1$. If $\{r_1\}$ forms a coalition with $\{z\}$, then every vertex of $L_1$ must be adjacent to $r_1$. Hence, $r_1$ is adjacent to all vertices of either $L_1$  or $R_2$. We show thirdly that $G[L_1]$ is a clique. Let $l_1\in L_1$. If $\{l_1\}$ forms a coalition with $\{x\}$, then $l_1 \in W$. Now, since $yz  \notin E$ and $z$ is not adjacent to $l_1$, $\{l_1\}$ does not form a coalition with $\{y\}$. If $\{l_1\}$ forms a coalition with $\{z\}$, then $l_1$ must be adjacent to all other vertices of $L_1$. Hence, $G[L_1]$ is a clique. From the above structural properties we infer that $G \in {\cal F}_2^3$ (see Definition \ref{def-f2}). Therefore, in this case when $yz \notin E$, we have $G\in {\cal F}_2$.

Hence we may assume that $yz \in E$, for otherwise the desired result follows. We show that every vertex of $R_2$ is adjacent to all vertices of $L_1$ or $R_2\setminus \{r_2\}$. Let $r_2\in R_2$. If $\{r_2\}$ forms a coalition with $\{x\}$, then $r_2 \in W$. If $\{r_2\}$ forms a coalition with $\{z\}$, then $r_2$ is adjacent to all vertices in $L_1$. If $\{r_2\}$ forms a coalition with $\{y\}$, then $r_2$ must be adjacent to all other vertices of $R_2$. Hence, we conclude that every vertex of $R_2$ is adjacent to all vertices of $L_1$ or $R_2\setminus \{r_2\}$. We show next that every vertex of $R_1$ is adjacent to all vertices of $L_1$ or $R_2$. Let $r_1 \in R_1$. If $\{r_1\}$ forms a coalition with $\{x\}$, then $r_1 \in W$. If $\{r_1\}$ forms a coalition with $\{z\}$, then all vertices of $L_1$ must be adjacent to $r_1$, and if $\{r_1\}$ forms a coalition with $\{y\}$, then  all vertices of $R_2$ must be adjacent to $r_1$. Hence, we conclude that every vertex of $R_1$ is adjacent to all vertices of $L_1$ or $R_2$. Finally, we show that every vertex of $L_1$ is adjacent to all vertices of $R_2$ or $L_1\setminus \{l_1\}$. Let $l_1\in L_1$. If $\{l_1\}$ forms a coalition with $\{x\}$, then $l_1 \in W$. If $\{l_1\}$ forms a coalition with $\{y\}$, then all vertices of $R_2$ must be adjacent to $l_1$, and if $\{l_1\}$ forms a coalition with $\{z\}$, then $l_1$ must be adjacent to all vertices $L_1\setminus \{l_1\}$. Hence, we conclude that every vertex of $L_1$ is adjacent to all vertices of $R_2$ or $L_1\setminus \{l_1\}$. From the above structural properties we infer that $G \in {\cal F}_2^3$ (see Definition \ref{def-f2}). Therefore, once again we have $G\in {\cal F}_2$.~\smallqed
\end{proof}

\begin{subclaim}
\label{c:claim3.2}
If $\{y\}$ and $\{z\}$ do not form a coalition, then $G \in {\cal F}_2$.
\end{subclaim}
\begin{proof}
Suppose that $\{y\}$ and $\{z\}$ do not form a coalition, implying that $L_2 \ne \emptyset$. We show that $L_2\subseteq W$. Suppose, to the contrary, that $L_2 \setminus W \ne \emptyset$. Let $l_2 \in L_2\setminus W$. Thus, $\{x\}$ and $\{l_2\}$ do not form a coalition, and so $\{l_2\}$ forms a coalition with $\{y\}$ or $\{z\}$. In both cases, since no vertex in $L_2$ is adjacent to $y$ or $z$, the vertex $l_2$ is adjacent to every other vertex in $L_2$. If $yz \notin E$, then since neither $y$ nor $z$ is adjacent to $l_2$, the set $\{l_2\}$ does not form a coalition with $\{y\}$ or with $\{z\}$, a contradiction. Hence, $yz \in E$. Let $q \in L_1\cup R_2\cup R_2$. If $\{q\}$ forms a coalition with $\{x\}$, then $q \in W$, and so $q$ is adjacent to $l_2$. If $\{q\}$ forms a coalition with $\{y\}$ or $\{z\}$, then since $l_2$ is adjacent to neither $y$ nor $z$, once again we infer that $l_2$ is adjacent to $q$. Hence, we conclude that $l_2$ adjacent to every vertex of $L_1 \cup R_1 \cup R_2$. As observed earlier, $l_2$ adjacent to every other vertex of $L_2$. Thus, $l_2$ adjacent to every other vertex of $V_x$, implying that $\{x\}$ and $\{l_2\}$ do form a coalition, and so $l_2 \in W$, a contradiction. Hence, $L_2 \subseteq W$. Proceeding analogously as in to the proof of Claim~\ref{c:claim3.1} when the set $L_2 = \emptyset$, we infer that $G \in {\cal F}_2^3$ (see Definition \ref{def-f2}), and so $G\in {\cal F}_2$.~\smallqed
\end{proof}

The proof of Claim~\ref{c:claim3} follows from Claims~\ref{c:claim3.2} and~\ref{c:claim3.2}.~\smallqed
\end{proof}

The proof of Theorem~\ref{thmf2} follows from Claims~\ref{c:claim1},~\ref{c:claim2} and~\ref{c:claim3}.~\QED
\end{proof}

\medskip
We next consider graphs $G$ with $\delta(G)=2$ that contain a full vertex.

\begin{theorem}
\label{thmfulldelta2}
If $G = (V,E)$ is a graph of order~$n$ with $\delta(G)=2$, then the following properties hold. \\[-24pt]
\begin{enumerate}
\item[{\rm (a)}] If $G$ contains exactly one full vertex, say $f$, then $C(G)=n$ if and only if $G[V \setminus\{f\}] \in {\cal F}_1$.
\item[{\rm (b)}] If $G$ contains exactly two full vertices, then $C(G)=n$ if and only if $G\cong (K_1\cup K_{n-3})+K_2$.
\item[{\rm (c)}] If $G$ contains at least three full vertices, then $G \cong C_3$.
\end{enumerate}
\end{theorem}
\begin{proof}
(a) Suppose that $G$ contains exactly one full vertex, say $f$. If $C(G) = n$, then since $f$ is a full vertex,  $C(G[V\setminus\{f\}])=n-1$.  Since $G$ has exactly one full vertex, $G[V\setminus\{f\}]$ has no full vertices. Since the minimum degree of $G[V\setminus\{f\}]$ is~$1$, by Theorem~\ref{thmdelta1nf} the graph $G[V\setminus\{f\}]$ belongs to the family~${\cal F}_1$. Conversely, if $G[V\setminus\{f\}] \in {\cal F}_1$, then $G$ has no full vertices and $C(G[V\setminus\{f\}])=n-1$. Adding $\{f\}$ to the singleton  coalition partition of  $G[V\setminus\{f\}]$ yields a singleton coalition  partition for $G$, whence $C(G)=n$.

(b) Suppose that $G$ contains exactly two full vertices, say $f_1$ and $f_2$. Let $G_{1,2} = G[V\setminus\{f_1,f_2\}]$. Since $\delta(G)=2$, the graph $G_{1,2}$ contains an isolated vertex, and so $G_{1,2}$ has minimum degree zero. If $C(G)=n$, then $C(G_{1,2})=n-2$, and so, by Theorem \ref{thmdelta0}, we have $G_{1,2} \cong K_1 \cup K_{n-3}$. Reconstructing the graph $G$ by adding back the full vertices $f_1$ and $f_2$ we have $G=(K_1\cup K_{n-3}) + K_2$. Conversely, if $G=(K_1\cup K_{n-3})+K_2$, then $C(G)=n$.

(c) Suppose that $G$ contains at least three full vertices, say $f_1$, $f_2$ and $f_3$. If $n \ge 4$, then $\delta(G) \ge 3$, contradicting our supposition that $\delta(G)=2$. Hence, $n = 3$, whence $G = C_3$.~\QED
\end{proof}

\medskip
Before we prove that if  $G\in {\cal F}_2$, then  $C(G)=n$, we will describe a family ${\cal H}_2$ of graphs. Thereafter, we establish a connection between $\CG(G, \Gamma_1)$ and the family ${\cal H}_2$, where $G \in {\cal F}_2$.

\begin{definition}[Family ${\cal H}_2$]
\label{def-h2}
The family  ${\cal H}_2$ consists of all graphs $H$ which are constructed as follows.

\begin{enumerate}
\item {\textbf{The family} $\mathbf{{\cal H}_2^1}$.}
{\rm
Let $H=(V,E)$ be the graph constructed as follows. Let $V = \{x',y',z'\} \cup R_1'$, with $R_1' \ne \emptyset$. Add all edges between vertices in $\{x',y',z'\}$ to form a clique, and so $H[\{x',y',z'\}]=K_3$. Join every  vertex of $R_1'$ to both $y'$ and $z'$. Let $R_1'$ be an independent set in $H$. Further, we add edges joining the vertex $x'$ to vertices in $R_1'$, including the possibility of adding no additional edge incident to $x'$ (except for the edges $x'y'$ and $x'z'$).  Let ${\cal H}_2^1$ be the family consisting of all such graphs $H$ constructed in this way. An example of a graph in the family ${\cal H}_2^1$ is illustrated in Figure~\ref{h2familyc1}.
}
\2

\item {\textbf{The family} $\mathbf{{\cal H}_2^2}$.}
{\rm
Let $H=(V,E)$ be the graph constructed as follows. Let $V=\{x',y',z'\} \cup R_1'\cup L_1'$, with $R_1'\ne \emptyset$ and $L_1' \ne \emptyset$. Join the vertex $y'$ to every vertex of $R_1'$, and join the vertex $z'$ to every vertex of $L_1'$.   We join the vertex $y'$ to both vertices $x'$ and $z'$. However, we do not add an edge between $x'$ and $z'$, and we add no edges between any vertex of $L_1'$ and $y'$. Let $L_1'\cup R_1'$ be an independent set of $H$. Additional edges may be added to $H$, including the possibility of none. By construction, every vertex in $L_1'$ has degree~$1$ with $z'$ as its unique neighbor or has degree~$2$ with $x'$ and $z'$ as its two neighbors. Let ${\cal H}_2^2$ be the family consisting of all such graphs $H$ constructed in this way. An example of a graph in the family ${\cal H}_2^2$ is illustrated in Figure~\ref{h2familyc2}.
}
\2

\begin{figure}[htb]
\begin{center}
{\subfloat[ \label{h2familyc1}]{\includegraphics[width = 0.32\textwidth]{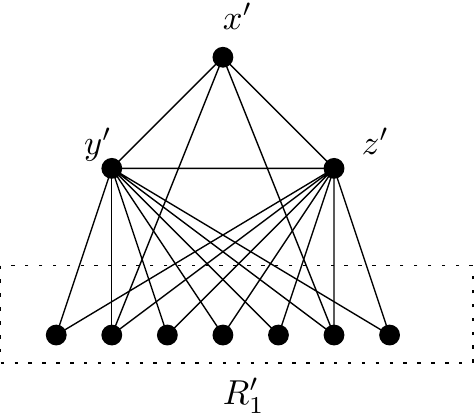}} }
\hspace{0.75cm}
{\subfloat[ \label{h2familyc2}]{\includegraphics[width =0.25\textwidth]{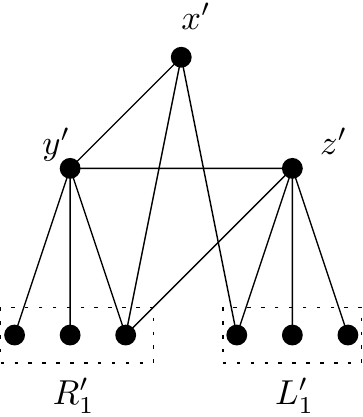}}}
\caption{(a): A graph in ${\cal H}_2^1$ (b): A graph in ${\cal H}_2^2$ }
\end{center}
\end{figure}

\item {\textbf{The family} $\mathbf{{\cal H}_2^3}$.}
{\rm
Let $H=(V,E)$ be the graph constructed as follows. Let $V = \{x',y',z'\} \cup L_1' \cup R_1' \cup R_2' \cup W'$, where $W'\ne \emptyset$ and where $L_1', R_1', R_2'$ and $W'$ are pairwise disjoint sets. Further some, including the possibility of all, of the sets $L_1'$, $R_1'$, and $R_2'$ may be empty. We add no edges between vertices in the set $L_1' \cup R_1' \cup R_2' \cup W'$, and so this set is an independent set in $H$. We do not add the edges $x'y'$ and $x'z'$. We join every vertex of $R_1'$ to at least one of the vertices $y'$ and $z'$. We join $x'$ to all vertices of $W'$, but we do not add any edge joining $x'$ to a vertex in the set $\{y',z'\} \cup L_1' \cup R_1'\cup R_2'$. We join each vertex of $L_1'$ and each vertex of $R_2'$ to at least one of the vertices $y'$ and~$z'$. Two examples of  graphs in the family ${\cal H}_{2}^3$ are illustrated in Figure \ref{h2familyc3}. 
%
%
%
}
\end{enumerate}
{\rm The family ${\cal H}_2$ is defined as ${\cal H}_2={\cal H}_2^1\cup {\cal H}_2^2\cup {\cal H}_2^3$. }
\end{definition}

\begin{figure}[ht]
\begin{center}
{\subfloat[ \label{h2familyc31}]{\includegraphics[width = 0.32\textwidth]{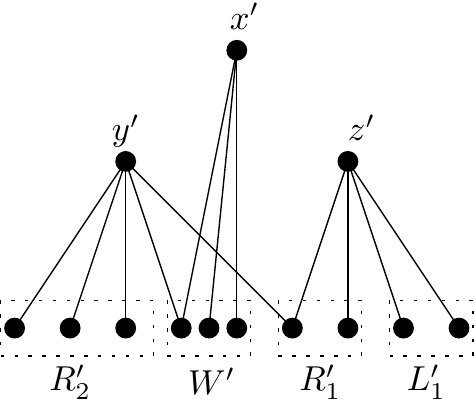}} }
\hspace{1cm}
{\subfloat[ \label{h2familyc32}]{\includegraphics[width =0.32\textwidth]{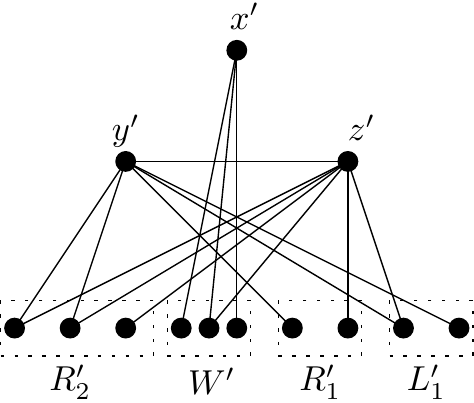}}}
\caption{ Two graphs in  ${\cal H}_{2}^3$.}
 \label{h2familyc3}
\end{center}
\end{figure}

We are now in a position to prove the following result.

\begin{theorem}
\label{lemh2f2}
If $G\in {\cal F}_2$, then $G$ is a \SP and $\CG(G,\Gamma_1) \in {\cal H}_2$.
\end{theorem}
\begin{proof}
Let $G\in {\cal F}_2$ have order~$n$, and let $G=(V,E)$.  We adopt our notation in Definition~\ref{def-f2} used to construct the graph $G$. Let $\{x,y,z\}\subseteq V$ and $N_G(x)=\{y,z\}$. In what follows, we let $V_x = V \setminus \{x,y,z \}$, and so $V = N_G[x] \cup V_x$.

If $G$ is a SP-graph, that is, if $C(G)=n$, then we let $H_2= \CG(G,\Gamma_1)$, where recall that $\Gamma_1$ represents a singleton coalition partition of $G$. In this case (when $C(G)=n$), by Observation~\ref{ob:ob1} for each vertex $u \in V \setminus\{x,y,z\}$, the set $\{u\}$ forms a coalition with at least one of the sets $\{x\}, \{y\}$, and $\{z\}$. Moreover, there is no coalition between any two sets $\{a\}$  and $\{b\}$ with $a,b \in V_x$. We also adopt our earlier notation that if $G$ is a SP-graph, then for a vertex $v$ in $G$, we represent the corresponding vertex in $H_2$ as~$\tilde{v}$.

By Definition~\ref{def-f2}, the family ${\cal F}_2$ is defined as ${\cal F}_2={\cal F}_2^1\cup {\cal F}_2^2\cup {\cal F}_2^3$. Thus, $G$ belongs to one of the families ${\cal F}_2^1$ or  ${\cal F}_2^2$ or ${\cal F}_2^3$. We proceed further with the following three claims.

\begin{claim}
\label{c:claimA}
If $G\in {\cal F}_2^1$, then $C(G)=n$ and $H_2 \in {\cal H}_2^1$.
\end{claim}
\begin{proof}
Suppose that $G \in {\cal F}_2^1$. In this case, both vertices $y$ and $z$ are adjacent to all vertices of $R_1$. If $u \in R_1$, then $\{u\}$ forms a coalition with both $\{y\}$ and $\{z\}$. Moreover, any two sets of $\{x\}$, $\{y\}$ and $\{z\}$ form a coalition. Hence, $C(G)=n$. We now consider the graph $H_2 = \CG(G,\Gamma_1)$. For each vertex $u \in R_1$, the vertex $\tilde{u}$ in $H_2$ is adjacent to both  $\tilde{y}$ and $\tilde{z}$. Moreover,  $H_2[\{\tilde{y}$, $\tilde{z}, \tilde{x}\}]$  is a clique. Thus, the graph $H_2$ belongs to the family ${\cal H}_2^1$ defined in Definition~\ref{def-h2}. Therefore, $H_2 \in {\cal H}_2$.~\smallqed
\end{proof}

\begin{claim}
\label{c:claimB}
If $G\in {\cal F}_2^2$, then $C(G)=n$ and $H_2 \in {\cal H}_2^2$.
\end{claim}
\begin{proof}
Suppose that $G \in {\cal F}_2^2$. Recall that in this case, $L_1 \ne \emptyset$ and $R_1\ne \emptyset$. Let $l_1 \in L_1$. Since $G[L_1]$ is a clique, the vertex $l_1$ is adjacent to every other vertex of $L_1$. Since the vertex $z$ is adjacent to all vertices in $R_1 \cup \{x\}$ and since the vertex $l_1$ dominates the set $L_1 \cup \{y\}$, the set $\{l_1,z\}$ is a dominating set of $G$, implying that the sets $\{l_1\}$ and $\{z\}$ form a coalition. Let $r_1 \in R_1$. Since the vertex $y$ is adjacent to every vertex of $L_1 \cup R_1 \cup \{x\}$, and since the vertex $r_1$ is adjacent to the vertex~$z$, the set $\{r_1,y\}$ is a dominating set of $G$, implying that the sets $\{r_1\}$ and $\{y\}$ form a coalition. Moreover since the vertex $x$ is adjacent to the vertex~$z$, the sets $\{x\}$ and $\{y\}$ form a coalition. Hence, $C(G)=n$.

We now consider the graph $H_2 = \CG(G,\Gamma_1)$. Since $\{x\}$ and $\{y\}$ form a coalition, $\tilde{x}$ is adjacent to  $\tilde{y}$ in $H_2$, and since $\{y\}$ and $\{z\}$ form a coalition, $\tilde{y}$ is adjacent to $\tilde{z}$. If $r_1\in R_1$, then $\tilde{r_1}$ is adjacent to $\tilde{y}$. If $l_1\in L_1$, then $\tilde{l_1}$ is adjacent to $\tilde{z}$. Since $z$ is adjacent to neither $y$ nor $l_1 \in L_1$, the set $\{y,l_1\}$ is not a dominating  set, implying that $\tilde{l_1}$ is not adjacent to $\tilde{y}$. Since $L_1 \ne \emptyset$, the set $\{x,z\}$ is not a dominating set, and so $\{x\}$ and $\{z\}$ do not form a coalition, implying that $\tilde{x}$ is not adjacent to $\tilde{z}$. If $r_1\in R_1$, then since there may exists edges between $L_1$ and $R_1$ in $G$, we note that it is possible that $\tilde{z}$ is adjacent to $\tilde{r_1}$. Thus, the graph $H_2$ belongs to the family ${\cal H}_2^2$ defined in Definition~\ref{def-h2}. Therefore, $H_2 \in {\cal H}_2$.~\smallqed
\end{proof}

\begin{claim}
\label{c:claimC}
If $G\in {\cal F}_2^3$, then $C(G)=n$ and $H_2 \in {\cal H}_{2}^3$.
\end{claim}
\begin{proof}
Suppose that $G \in {\cal F}_2^3$. Recall that by definition, $L_1 \ne \emptyset$, $R_2 \ne \emptyset$, and $W \ne \emptyset$, and that if $L_2 \ne \emptyset$, then $L_2 \subseteq W$. Further, $W$ may intersects $L_1\cup R_1\cup R_2\cup L_2$. Let $w \in W$. Since $G[W]$ is a clique, and $w$ is adjacent to all vertices $L_1 \cup R_1 \cup R_2\cup L_2$, the set $\{w\}$ forms a coalition with $\{x\}$. We note that $W$ may intersects $L_1 \cup R_1 \cup R_2 \cup L_2$.  Let $R_1 \ne \emptyset$, and let $r_1\in R_1$. By definition, the vertex $r_1$ is adjacent to all vertices of either $L_1$ or $R_2$. If $r_1$ is adjacent to all vertices of $L_1$, then $\{r_1\}$ and $\{z\}$ form a coalition, and if $r_1$ is adjacent to all vertices of $R_2$, then $\{r_1\}$ and $\{y\}$ forms a coalition.

Suppose firstly that $yz \notin E(G)$. In this case, both $G[L_1]$ and $G[R_2]$ are cliques. Therefore, if $l_1\in L_1$, then $\{z,l_1\}$ is a dominating set of $G$, and so $\{l_1\}$ and $\{z\}$ form a coalition. Moreover if $r_2\in R_2$, then $\{y,r_2\}$ is a dominating set of $G$, and so $\{r_2\}$ and $\{y\}$ form a coalition.

Suppose secondly that $yz \in E(G)$. In this case, every vertex $l_1\in L_1$ is adjacent to all vertices of either $L_1$ or $R_2$, and  every vertex $r_2\in R_2$ is adjacent to all vertices of either $L_1$ or $R_2$. If $l_1\in L_1$ and $l_1$ is adjacent to all vertices of $L_1$, then  $\{l_1\}$ and $\{z\}$ form a coalition, and if $l_1$ is adjacent to all vertices of $R_2$, then $\{l_1\}$ and $\{y\}$ form a coalition. Analogously, if $r_2 \in R_2$, then $\{r_2\}$ forms a coalition with either $\{z\}$  or $\{y\}$.

From the above properties we infer that $C(G)=n$. We now consider the graph $H_2 = \CG(G,\Gamma_1)$. Since $L_1\ne \emptyset $ and $R_2\ne \emptyset$, the set $\{x\}$ does not form a coalition with $\{y\}$ or $\{z\}$. Thus, $\tilde{x}$ is adjacent to neither $\tilde{y}$ nor $\tilde{z}$ in $H_2$. If $w \in W$, then the vertex set $\{x\}$ forms a coalition with $\{w\}$, and so $\tilde{w}$ is adjacent to $\tilde{x}$. If $R_1 \ne \emptyset$ and $r_1\in R_1$, then $\{r_1\}$ forms a coalition with either $\{y\}$ or $\{z\}$, and so $\tilde{r_1}$ is adjacent to either $\tilde{y}$ or $\tilde{z}$.
%
%
If $yz \notin E(G)$, then for $l_1\in L_1$ and $r_2\in R_2$, we know that $\{l_1\}$ forms a coalition with $\{z\}$, and $\{r_2\}$ forms a coalition with $\{y\}$. Hence, $\tilde{l_1}$ is adjacent to $\tilde{z}$, and $\tilde{r_2}$ is adjacent to $\tilde{y}$. 
%
If $yz \in E(G)$, then for $l_1\in L_1$ and $r_2\in R_2$, we know that each of $\{l_1\}$ and $\{r_2\}$ form a coalition with either $\{y\}$ or $\{z\}$. Hence, both $\tilde{l_1}$ and $\tilde{r_2}$ are adjacent to either $\tilde{z}$ or  $\tilde{y}$.
Let $W'=W, L_1'=L_1\setminus W, R_1'=R_1\setminus W,  R_2'=R_2\setminus W$, and $x'=\tilde{x}, y'=\tilde{y}$, and $z'=\tilde{z}$. If $u\in W\setminus L_1\cup R_1\cup R_2$, then $\{u\}$ and $\{x\}$ do not form a coalition, and therefore $x'$ is not adjacent to any any vertex of $L_1'\cup R_1'\cup R_2'$ in $H_2$. 
{ Note that if $L_2=\emptyset$, then $\tilde{y}$ is adjacent to $\tilde{z}$ in $H_2$, and if $L_2\ne \emptyset$, since there is a vertex in $L_2$ which has no neighbor in  $\{y,z\}$,  $\tilde{y}$ is not adjacent to $\tilde{z}$ in $H_2$}.
Thus, the graph $H_2$ belongs to the family ${\cal H}_{2}^3$ defined in Definition~\ref{def-h2}. Therefore, $H_2 \in {\cal H}_2$.~\smallqed
\end{proof}

\medskip
The proof of Theorem~\ref{lemh2f2} follows from Claims~\ref{c:claimA},~\ref{c:claimB} and~\ref{c:claimC}.~\QED
\end{proof}

\section{$G$-SC chain}

In this section, we characterize the $G$-SC chain for all \SPs $G$ with $0\le \delta(G)\le 2$.

\subsection{$G$-SC chain  with $\delta(G)=0$}
\label{sec0}

We first consider the case when the  \SP $G$ contains an isolated vertex.  We shall prove the following theorem.

\begin{theorem}
\label{thmlscc0}
If $G$ is an \SP of order $n$ with $\delta(G)=0$, then the following properties hold.
\\[-24pt]
\begin{enumerate}
\item[{\rm (a)}] If $n=1$, then $G=K_1$,  $L_{\SCC}(G)=0$, and the $G$-SC chain is $K_1\rightarrow K_1\rightarrow K_1\rightarrow \cdots$ \1
\item[{\rm (b)}] If $n=2$, then $G=\overline{K}_2$,  $L_{\SCC}(G)=\infty$, and the $G$-SC chain is $\overline{K}_2 \rightarrow  K_2\rightarrow \overline{K}_2\rightarrow K_2 \rightarrow  \cdots$ \1
\item[{\rm (c)}] If $n>3$,  then $G\cong K_1\cup K_{n-1}$,  $L_{\SCC}(G)=1$, and the $G$-SC chain is $G\rightarrow K_{1,n-1}$. \1
\item[{\rm (d)}] If $n=3$, then $G=K_1\cup K_2$,   $L_{\SCC}(G)=\infty$, and the $G$-SC chain is
\[
K_1\cup K_2\rightarrow  P_3\rightarrow K_1\cup K_2\rightarrow P_3\rightarrow  \cdots
\]
\end{enumerate}
\end{theorem}
\begin{proof}
Let $G$ be a \SP of order $n$ with $\delta(G)=0$. If $n=1$, then $G=K_1$. In this case, the $K_1$-SC chain is given by the sequence
$K_1 \rightarrow K_1 \rightarrow K_1 \rightarrow \cdots$. This proves part~(a).

(b) If $n=2$, then $G=\overline{K}_2$. In this case, $\CG(\overline{K}_2,\Gamma_1) = K_2$. Since $\CG(K_2,\Gamma_1)=\overline{K}_2$, the $\overline{K}_2$-SC chain is therefore given by the sequence
\[
\overline{K}_2\rightarrow  K_2\rightarrow \overline{K}_2\rightarrow K_2\rightarrow  \cdots.
\]

(c) and (d): Suppose that $n \ge 3$. By Theorem \ref{thmdelta0},  $G \cong K_1\cup K_{n-1}$.  We now consider the graph $\CG(G,\Gamma_1)$. Let $q$ be the isolated vertex of $G$, and $\{x_1,\ldots,x_{n-1}\}$ be the vertex set of $K_{n-1}$.  Thus, the set $\{x_i\}$ forms a coalition with $\{q\}$, and so $\tilde{x_i}$ is adjacent to $\tilde{q}$ in $\CG(G,\Gamma_1)$  for all $i \in [n-1]$. By Observation~\ref{ob:ob1}, there is no coalition between the sets of $\{x_i\}$ and $\{x_j\}$ for $1 \le i < j \le n-1$. Hence, $\CG(G,\Gamma_1)$ is the star $K_{1,n-1}$. We now consider the graph $K_{1,n-1}$, which has minimum degree~$1$ and contains exactly one full vertex. By Theorem \ref{lemdelta1full}, the graph $K_{1,n-1}$ is a \SP if and only if $n=3$. Hence if $n > 3$, then the $G$-SC chain is given by the sequence $G \rightarrow K_{1,n-1}$, which proves part~(c).
Suppose, finally, that $n=3$. In this case, $G=K_1\cup K_2$ and $\CG(G,\Gamma_1)=K_{1,2}\cong P_3$. Since $\CG(P_3, \Gamma_1)=K_1\cup K_2$, the $G$-SC chain for $G=K_1\cup K_2$ is given by the sequence
\[
K_1\cup K_2\rightarrow  P_3\rightarrow K_1\cup K_2\rightarrow P_3\rightarrow  \cdots
\]
which completes the proof of part~(d), and of Theorem~\ref{thmlscc0}.~\QED
\end{proof}

\subsection{$G$-SC chain  with $\delta(G)=1$}

In this section, we consider the case when the \SP $G$ has minimum degree~$1$.  We first consider the case when $G$ contains a full vertex.

\begin{theorem}
\label{thby}
If $G$ is an \SP of order $n$ with $\delta(G)=1$ that contains a full vertex, then the following properties hold.
\\[-24pt]
\begin{enumerate}
\item[{\rm (a)}] If $n=2$, then $G=K_2$,   $L_{\SCC}(G)=\infty$, and the $G$-SC chain is $K_2\rightarrow \overline{K}_2 \rightarrow K_2\rightarrow \overline{K}_2\ldots$ \1
\item[{\rm (b)}] If $n>3$ and $L_{\SCC}(G)=1$, then the $G$-SC chain is $G\rightarrow K_1 \cup K_{1,n-2}$. \1
\item[{\rm (c)}] If  $n=3$, then $G=P_3$, $L_{\SCC}(G)=\infty$, and the $G$-SC chain is
\[
P_3\rightarrow K_1\cup K_2\rightarrow P_3\rightarrow K_1\cup K_2\ldots
\]
\end{enumerate}
\end{theorem}
\begin{proof}
Let $G$ be a \SP of order $n$ with $\delta(G)=1$ that $G$ contains at least one full vertex. If $n=2$, then $G=K_2$. In this case since $\CG(K_2,\Gamma_1)=\overline{K}_2$, the $K_2$-SC chain is given by the sequence $K_2\rightarrow \overline{K}_2 \rightarrow K_2\rightarrow \overline{K}_2 \cdots$. This proves part~(a).

(b) and (c): Suppose that $n \ge 3$. In this case, the graph $G$ contains exactly one full vertex. By Theorem~\ref{lemdelta1full}, $G$ is obtained from $K_1\cup K_{n-1}$  by adding an edge between the isolated vertex, say $q$,  and a vertex, say $x$, in the complete graph $K_{n-1}$. We now consider $\CG(G,\Gamma_1)$. By Observation~\ref{ob:ob1}, if $y$ is an arbitrary vertex of the complete graph $K_{n-1}$ different from~$x$, then the set $\{y\}$ forms a coalition with $\{q\}$, and so $\tilde{y}$ is adjacent to $\tilde{q}$ in $\CG(G,\Gamma_1)$. Moreover, no two vertices in the complete graph $K_{n-1}$, different from~$x$, form a coalition. The full vertex~$x$ does not form a coalition with any other set. Therefore, $\CG(G,\Gamma_1) \cong K_1 \cup K_{1,n-2}$. However, $K_1 \cup K_{1,n-2}$ is a \SP if and only if $n=3$.  Hence if $n > 3$, then the $G$-SC chain is given by the sequence $G \rightarrow K_1 \cup K_{1,n-2}$, which proves part~(b).
Suppose, finally, that $n=3$. In this case, $G \cong P_3$ and  $\CG(P_3,\Gamma_1)=K_1\cup K_2$. Since $\CG(K_1\cup K_2, \Gamma_1)=P_3$, the $G$-SC chain for $G$ is given by the sequence
\[
P_3 \rightarrow K_1\cup K_2\rightarrow P_3\rightarrow K_1\cup K_2\ldots
\]
which completes the proof of part~(c), and of Theorem~\ref{thby}.~\QED
\end{proof}

We first next the case when  the \SP $G$ contains no full vertex.

\begin{theorem}
\label{thbyr}
If $G$ is a \SP of order $n$ with $\delta(G)=1$ that contains no full vertex, then the $G$-SC chain is one of the following chains, where $H_1$ is a graph in the family ${\cal H}_1$. \\[-24pt]
\begin{enumerate}
\item[{\rm (a)}] $G\rightarrow H_1$,
\item[{\rm (b)}] $G\rightarrow C_4\rightarrow K_4\rightarrow \overline{K}_4$, or
\item[{\rm (c)}] $G\rightarrow K_{2,n-2} \rightarrow K_1 + {K}_{1,n-2}$.
\end{enumerate}
\end{theorem}
\begin{proof}
Let $G$ be a \SP of order $n$ with $\delta(G)=1$ that contains no full vertex. By Theorem~\ref{thmdelta1nf}, $G\in {\cal F}_1$.  By Theorem \ref{thmfh1cg}, $H_1= \CG(G,\Gamma_1)\in {\cal H}_1$.  If $H_1$ is not a SP-graph, then the $G$-SC chain is the chain $G\rightarrow H_1$. Hence, we may assume that $H_1$ is a SP-graph. We now compute the SC-graph of $H_1$.  By construction, all graphs in the family ${\cal H}_1$ have order at least~$4$, implying that $n \ge 4$.

We show next that $Q_1 = \emptyset$.  If $H_1\cong P_4$, then it is obvious that $Q_1=\emptyset$. Now, we assume that $H_1$ is not isomorphic to $P_4$. Suppose, to the contrary, that $Q_1 \ne \emptyset$. In this case, $|Q_1| \ge 2$ and the graph $H_1$ is illustrated in Figure~\ref{figH}(a) and has vertex set $\{x_1,y_1\} \cup P_1 \cup \{w_1\} \cup Q_1$. Every vertex in $Q_1$ has degree~$1$ in $H_1$, and so $H_1$ has at least $|Q_1| \ge 2$ vertices of degree~$1$. Since every   connected graph in the family~${\cal F}_1\setminus \{P_4\}$ has exactly one vertex of degree~$1$, namely the vertex~$x$ in Definition~\ref{defn1}, the graph $H_1 \notin {\cal F}_1$. Thus by Theorem~\ref{thmdelta1nf}, the graph $H_1$ is not a SP-graph, a contradiction to our assumption.

Hence, $Q_1 = \emptyset$. In this case, the graph $H_1$ is illustrated in Figure~\ref{figH}(b). By the construction of graphs in the family ${\cal H}_1$ with $Q_1 = \emptyset$, we have $H_1\cong K_{2, n-2}$. If $n=4$, then $G\cong C_4$. As observed earlier, the $C_4$-SC chain is given by $C_4 \rightarrow K_4 \rightarrow \overline{K}_4$. Hence, we may assume that $n \ge 5$.

Adopting our earlier notation when constructing graphs in the family ${\cal H}_1$, we let $B_1 = P_1 \cup \{w_1\}$. Further, we let $A_1 = \{x_1,y_1\}$. Thus, $H_1$ is the complete bipartite graph $K_{2, n-2}$ with partite sets $A_1$ and $B_1$, where $|A_1| = 2$ and $|B_1| = n - 2 \ge 3$. If $a \in A_1$ and $b \in B_1$, then the sets $\{a\}$ and $\{b\}$ form a coalition in $H_1$. Moreover, the sets $\{x_1\}$ and $\{y_1\}$ form a coalition in $H_1$. However, if $b_1$ and $b_2$ are two distinct vertices in $B_1$, then since $n \ge 5$, the sets $\{b_1\}$ and $\{b_2\}$ do not form a coalition. Hence, the singleton coalition graph of $H_1$ is isomorphic to $K_1+K_{1,n-2}$, where here `$+$' denotes the join operation. However, since the graph $K_1+K_{1,n-2}$ has two full vertices, the coalition graph of $K_1+K_{1,n-2}$ associated with any coalition partition  has two isolated vertices, and is therefore not a SP-graph. Hence, in this case when $n \ge 5$, the $G$-SC chain is given by $G\rightarrow K_{2,n-2} \rightarrow K_1+K_{1,n-2}$.~\QED
\end{proof}

\subsection{$G$-SC chain  with $\delta(G)=2$}

In this section, we consider the case when the \SP $G$ has minimum degree~$2$.  We first prove the following result.

\begin{theorem}
\label{thmlscc21}
If $G$ is a \SP of order $n$ with $\delta(G)=2$ that contains at least one full vertex, then $L_{\SCC}(G)=1$.
\end{theorem}
\begin{proof}
Suppose that $G = (V,E)$ is a \SP of order $n$ with $\delta(G)=2$ that contains at least one full vertex. We note that $n \ge 3$. If $n = 3$, then $G = K_3$. In this case, $\CG(G, \Gamma_1) = \overline{K}_3$. Since $\overline{K}_3$ is not a SP-graph, the $G$-SC chain is the chain $K_3 \rightarrow \overline{K}_3$, and so  $L_{\SCC}(G)=1$. Hence we may assume that $n \ge 4$, for otherwise the desired result follows. Thus, $G$ has at most two full vertices. Let $I = \CG(G, \Gamma_1)$. If $G$ has two full vertices, then $I$ has two isolated vertices, and so, by Theorem~\ref{thmdelta0}, the graph $I$ is not a SP-graph, implying that in this case the $G$-SC chain is the chain $G \rightarrow I$, and so  $L_{\SCC}(G)=1$. Hence, we may assume that $G$ has exactly one full vertex, say $f$. By Theorem~ \ref{thmfulldelta2}, $G' = G[V \setminus\{f\}] \in {\cal F}_1$. Let $H' = \CG(G',\Gamma_1)$. By Theorem~\ref{thmfh1cg}, $H' \in {\cal H}_1$. Since $G$ has a full vertex, the singleton coalition graph $I$ of $G$ contains an isolated vertex, implying that $I \cong K_1 \cup H'$. If $I$ is a SP-graph, then by Theorem~\ref{thmdelta0}, $H' \cong K_{n-1}$. However by the definition of family ${\cal H}_1$ (see Definition~\ref{defnH1}), the graph $H'$ is not isomorphic to $K_{n-1}$, a contradiction. Hence, $I$ is not a SP-graph, implying that the $G$-SC chain is once again the chain $G \rightarrow I$, and so  $L_{\SCC}(G)=1$.~\QED
\end{proof}

\medskip
Suppose, next, that $G$ is a \SP with $\delta(G)=2$ that contains no full vertex. Let $B = \CG(G,\Gamma_1)$. By Theorem~\ref{lemh2f2}, $B \in {\cal H}_2$. If $B$ is not a SP-graph, then the $G$-SC chain is $G\rightarrow B$, and so $L_{\SCC}(G)=1$. Hence we may assume in what follows that $B$ is an SP-graph. By Definition~\ref{def-h2}, the singleton coalition graph $B \in {\cal H}_2$ belongs to one of the three families ${\cal H}_2^1$, ${\cal H}_2^2$, or ${\cal H}_2^3$. We consider the three possibilities in turn for the graph $B$. Throughout the following three lemmas, we adopt our notation in  Definition~\ref{def-h2}.

\begin{lemma}
\label{lemcase1}
If $B \in {\cal H}_2^1$, then the $G$-SC chain is one of the following chains. \\ [-24pt]
\begin{enumerate}
\item[{\rm (a)}]$G\rightarrow B\rightarrow \overline{K}_4$, 
\item[{\rm (b)}]$G\rightarrow B\rightarrow \overline{K}_3\cup K_2$,
\item[{\rm (c)}] $G \rightarrow B \rightarrow \overline{K}_2\cup K_2$ or
\item[{\rm (d)}] $G\rightarrow B\rightarrow \overline{K}_2\cup P_3$.
\end{enumerate}
\end{lemma}
\begin{proof}
Suppose that the singleton coalition graph $B \in {\cal H}_2^1$ (as illustrated in Figure~\ref{h2familyc1}).  In this case, the graph $B$ has either  two full vertices (namely $y'$ and $z'$) or three full vertices (namely $x', y'$, and $z'$).    Thus, the singleton sets $\{z'\}$ and $\{y'\}$  (and also singleton set $\{x'\}$ when $B$ has three full vertices) are included in any coalition partition of $B$. Let $B'$ be the graph obtained from $B$ by removing all  its full vertices.  Let $B'$ have order~$n'$, and so $n' = n-2$ or $n'=n-3$. We note that $B$ is a \SP if and only if $B'$ is a SP-graph.

If $n'=n-3$, then  $B'\cong\overline{K}_{n'}$. It is clear that $B'$ is a \SP if and only if $B'\cong\overline{K}_{1}$ or $B'\cong\overline{K}_{2}$. Then, $B\cong K_4$ or $B\cong\overline{K_2}+K_3$. Hence, the $G$-SC chain is the sequence
$$G\rightarrow B\rightarrow \overline{K}_4,~\text{if~} B\cong K_4,$$
or  the sequence
$$G\rightarrow B\rightarrow \overline{K}_3\cup K_2,~\text{if~} B\cong\overline{K_2}+K_3.$$

Now, we assume that $n'=n-2$. By construction of graphs in the family~${\cal H}_2^1$, either the vertex~$x'$ has degree~$2$ in $B$ or the vertex~$x'$ has degree at least~$3$ in $B$. Further if $x'$ has degree~$2$ in $B$, then every vertex in the independent set $R_1'$ has degree~$2$ and is adjacent to only $y'$ and $z'$ in $B$, while if $x'$ has degree at least~$3$ in $B$, then we can partition the set $R_1'$ into two sets $R_{1,2}'$ and $R_{1,3}'$ where every vertex in $R_{1,2}'$ has degree~$2$ and is adjacent to only $y'$ and $z'$ in $B$, and where every vertex in $R_{1,3}'$ has degree~$3$ and is adjacent to only $x'$, $y'$ and $z'$ in $B$. Hence if $x'$ has degree~$2$ in $B$, then $B' = \overline{K}_{n'}$, while if $x'$ has degree at least~$3$ in $B$, then $B' = \overline{K}_r \cup K_{1,n'-r-1}$ for some integer $r \ge 1$ where in this case $n' \ge r + 2$. Note that since $n'=n-2$, we must have $r\geq 1$.  In both cases, the graph $B'$ contains at least one isolated vertex.

By our earlier observations, $B$ is a \SP if and only if $B'$ is a SP-graph. Furthermore, the graph $B'$ contains at least one isolated vertex. By Theorem~\ref{thmdelta0}, $B'$ is a \SP if and only if $B' \cong K_1 \cup K_{n'-1}$. As observed earlier, either $B' = \overline{K}_{n'}$ or $B' = \overline{K}_r \cup K_{1,n'-r-1}$ for some integer $r \ge 1$ where in this case $n' \ge r + 2$. Therefore, $B'$ is a \SP if and only if   $B'\cong \overline{K}_2$ or $B'\cong K_1\cup K_2$.

Reconstructing the graph $B$ from the graph $B'$ by adding back the deleted vertices $y'$ and $z'$, we infer that $B$ is a \SP if and only if it is isomorphic to the graph $K_4 - e$ obtained from a $K_4$ by removing an edge (illustrated in Figure~\ref{figHpg}(a)) or the graph of order~$5$ obtained from $K_4$ by adding a new vertex and joining it to two vertices of the complete graph (illustrated in Figure~\ref{figHpg}(b)). Moreover, if $B = K_4 - e$, then  $\CG(B,\Gamma_1) \cong \overline{K}_2\cup K_2$, while if $B$ is obtained from $K_4$ by adding a new vertex of degree~$2$, then $\CG(B,\Gamma_1) \cong \overline{K}_2 \cup P_3$. In both cases, $\CG(B, \Gamma_1)$ is not a SP-graph. Therefore, the $G$-SC chain is given by $G \rightarrow B \rightarrow \overline{K}_2\cup K_2$ or by $G\rightarrow B\rightarrow \overline{K}_2\cup P_3$.~\QED
\end{proof}

\begin{figure}[htb]
\begin{center}
{\subfloat[]{\includegraphics[width = 0.2\textwidth]{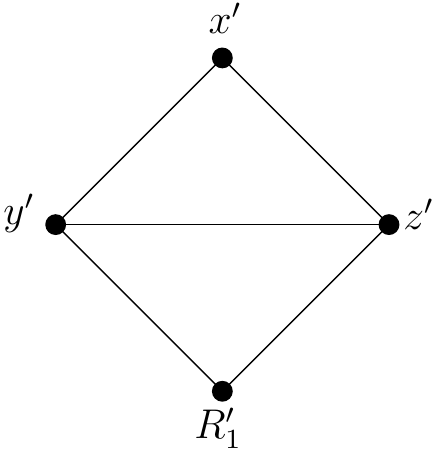}} }
\hspace{1.25cm}
{\subfloat[]{\includegraphics[width =0.2\textwidth]{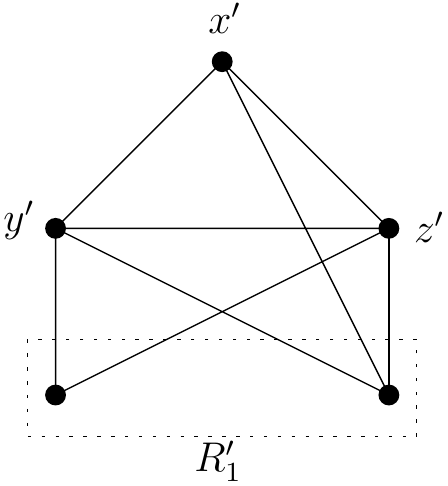}}}
\caption{(a): The graph $B$ with $B'\cong \overline{K}_2.$ (b). The graph $B$ with $B'\cong K_1\cup K_2.$}
\label{figHpg}
\end{center}
\end{figure}

Let $M_1, M_2$, and  $M_3$   be three  graphs depicted in Figure \ref{cs}. Moreover, let $M_4$ be obtained from $M_3$ by adding an edge from one of the leaf neighbors of the vertex of degree~$n-2$ in $M_3$ to the vertex of degree~$2$ in $M_3$.

\begin{figure}[htb]
\begin{center}
{\subfloat[ \label{figM1}]{\includegraphics[width = 0.15\textwidth]{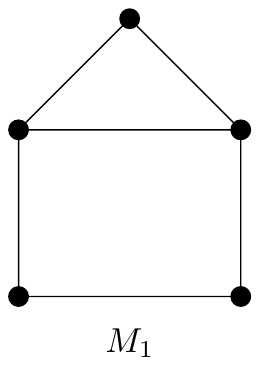}} }
\hspace{1cm}
{\subfloat[\label{figM2}]{\includegraphics[width = 0.15\textwidth]{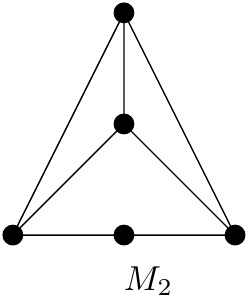}}}
\hspace{1cm}
{\subfloat[\label{figM3}]{\includegraphics[width = 0.25\textwidth]{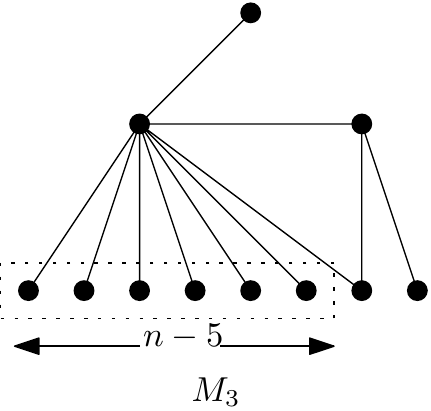}}}
\caption{(a): The graph $M_1$. (b): The graph $M_2$.  (c): The graph $M_3$.}
\label{cs}
\end{center}
\end{figure}

We are now in a position to present the following lemma.

\begin{lemma}
\label{lemcase2}
If $B \in {\cal H}_2^2$, then the $G$-SC chain is one of the following chains, where $H_1$ is a graph in the family ${\cal H}_1$. \\ [-24pt]
\begin{enumerate}
\item[{\rm (a)}] $G\rightarrow B \rightarrow H_1$,
\item[{\rm (b)}] $G\rightarrow B \rightarrow C_4\rightarrow K_4 \rightarrow \overline{K}_4$.
\item[{\rm (c)}] $G\rightarrow B \rightarrow K_{2,n-2} \rightarrow K_1 + K_{1,n-2}$,
\item[{\rm (d)}] $G \rightarrow M_1 \rightarrow \left(\, \overline{K}_2+(K_1\cup K_2)\right)\rightarrow  \overline{K}_2+K_3\rightarrow \overline{K}_3\cup K_2$,
\item[{\rm (e)}] $G\rightarrow M_2\rightarrow  \overline{K}_2+K_3\rightarrow \overline{K}_3\cup K_2$,
\item[{\rm (f)}] $G \rightarrow B \rightarrow K_3 \circ K_1$,
\item[{\rm (g)}] $G\rightarrow B\rightarrow M_3$,
\item[{\rm (h)}] $G\rightarrow B\rightarrow M_4$,
\item[{\rm (i)}] $G\rightarrow B \rightarrow K_2 + \overline{K}_{n-2}$, or
\item[{\rm (j)}] $G \rightarrow B \rightarrow (K_2 + \overline{K}_{n-2}) + e$.
\end{enumerate}
\end{lemma}
\begin{proof}
Suppose that the singleton coalition graph $B \in {\cal H}_2^2$ (as illustrated in Figure~\ref{h2familyc2}). In this case, $L_1'\ne \emptyset$ and $R_1'\ne \emptyset$. Further the set $L_1'\cup R_1'$ is an independent set and every vertex of $L_1'$ is adjacent to $z'$ and possibly adjacent to $x'$. By construction, we have $n(B) \ge 5$ and  $1 \le \delta(B)\le 2$. By assumption, $B$ is an SP-graph.

Suppose firstly that $\delta(B)=1$.  Based on the construction of $B$, there is no full vertex in $B$. Thus by Theorem~\ref{thbyr}, the $B$-SC chain is given by the chain in part~(a),~(b) or~(c) in the statement of Theorem~\ref{thbyr}, where $H_1$ is a graph in the family ${\cal H}_1$. Thus, the $G$-SC chain is given by part~(a),~(b) or~(c) in the statement of Lemma~\ref{lemcase2}.

Hence, we may assume that $\delta(B)=2$, for otherwise the desired result of the lemma follows. Recall that by construction, the vertex $y'$ is not adjacent to any vertex of $L_1'$ but is adjacent to every vertex of $R_1'$. Hence in this case when $\delta(B) = 2$, the vertex $x'$ must be adjacent to all vertices of $L_1'$ (see Figure \ref{h2familyc2}), and so every vertex in $L_1'$ has degree~$2$ in $B$ (with $x'$ and $z'$ as its two neighbors). Further, every vertex in $R_1'$ is adjacent to~$y'$ and to at least one of $x'$ and $z'$.

\begin{claim}
\label{c:claimH22a}
If $|L_1'|=|R_1'|=1$, then the $G$-SC chain is given by part~(d) or~(e) in the statement of Lemma~\ref{lemcase2}.
\end{claim}
\begin{proof}
Suppose that $|L_1'|=|R_1'|=1$. Let $R_1' = \{r_1\}$ and $L_1' = \{z_1\}$. By our earlier observations, the vertex $z_1$ has degree~$2$ with $x'$ and $z'$ as its two neighbors. Further, the vertex $r_1$ is adjacent to~$y'$ and to at least one of $x'$ and $z'$. If $r_1$ is adjacent to exactly one of $x'$ and $z'$, then $B\cong M_1$. If $r_1$ is adjacent to both $x'$ and $z'$, then $B=M_2$. In both cases, $B$ is a SP-graph. If $B\cong M_1$, then the $G$-SC chain is the sequence
\[
G \rightarrow M_1 \rightarrow \left(\, \overline{K}_2+(K_1\cup K_2)\right)\rightarrow \overline{K}_2+K_3\rightarrow \overline{K}_3\cup K_2,
\]
while if $B\cong M_2$,  then the $G$-SC chain is the sequence
\[
G \rightarrow M_2\rightarrow  \overline{K}_2+K_3\rightarrow \overline{K}_3\cup K_2.
\]

Thus in this case when $|L_1'|=|R_1'|=1$, the $G$-SC chain is given by part~(d) or~(e) in the statement of Lemma~\ref{lemcase2}.~\smallqed
\end{proof}

\medskip
By Claim~\ref{c:claimH22a}, we may assume that $|L_1'| \ge 2$ or $|R_1'| \ge 2$, for otherwise the desired result of the lemma follows. Thus, $L_1' \cup R_1'$ is an independent set of cardinality at least~$3$ in $B$. Hence if $u$ and $v$ are two arbitrary vertices in $L_1' \cup R_1'$, the sets $\{u\}$ and $\{v\}$ do not form a coalition.

Let $r_1 \in R_1'$. Since $y'$ is not adjacent to any vertex of $L_1'$, the set $\{r_1\}$ forms a coalition with only $\{x'\}$ or $\{z'\}$.  If $\{r_1\}$ forms a coalition with $\{x'\}$, then $z'$ must be adjacent to $r_1$, and $x'$ must be adjacent to all vertices of $R_1' \setminus \{r_1\}$. If $\{r_1\}$ forms a coalition with $\{z'\}$, then $x'$ must be adjacent to $r_1$, and $z'$ must be adjacent to all vertices of $R_1 \setminus\{r_1\}$. This property holds for all vertices in $R_1'$. Thus if $|R_1'| \ge 3$, then we infer that all edges between $R_1'$ and $\{x',z'\}$ are present, except for possibly one edge. If $|R_1'| = 2$, then we infer that both $x'$ and $z'$ are adjacent to at least one vertex in $R_1'$ and, as observed earlier, every vertex in $R_1'$ is adjacent to at least one of $x'$ and $z'$. In particular, when $|R_1'| = 2$, we infer that all edges between $R_1'$ and $\{x',z'\}$ are present, except for possibly at most two edges.

\begin{claim}
\label{c:claimH22b}
If all edges between $R_1'$ and $\{x',z'\}$ are present, except for exactly two edges, then the $G$-SC chain is given by part~(f) in the statement of Lemma~\ref{lemcase2}.
\end{claim}
\begin{proof}
Suppose that all edges between $R_1'$ and $\{x',z'\}$ are present, except for exactly two edges. In this case, $R_1' = \{r_1,r_2\}$ and, renaming vertices if necessary, we may assume that $r_1$ is adjacent to $x'$ but not to $z'$, and $r_2$ is adjacent to $z'$ but not to $x'$. Thus, $\{r_1\}$ forms a coalition with $\{z'\}$, while $\{r_2\}$ forms a coalition with $\{x'\}$. If $|L_1'| \ge 2$ and $z_1 \in L_1'$, then the set $\{z_1\}$ cannot form a coalition with any other singleton set in $B$, contradicting our assumption that $B$ is an SP-graph. Hence, $|L_1'| = 1$. Let $L_1' = \{z_1\}$. We now infer that $\{z_1\}$ forms a coalition with only the set $\{y'\}$, the set $\{x'\}$ forms a coalition with each of the sets $\{y'\}$ and $\{z'\}$, and the set $\{y'\}$ and $\{z'\}$ form a coalition. However, there are no additional coalitions in $B$. Hence in this case, the singleton coalition graph of $B$ is given by $\CG(B,\Gamma_1) \cong K_3 \circ K_1$, where $K_3 \circ K_1$ denotes the corona of $K_3$ as illustrated in Figure~\ref{corona}. However, $K_3 \circ K_1$ is not a SP-graph, and so in this case the $G$-SC chain is the sequence $G \rightarrow B \rightarrow K_3 \circ K_1$ given by part~(f) in the statement of Lemma~\ref{lemcase2}.~\smallqed
\end{proof}

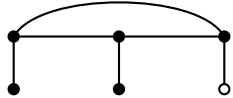
\begin{figure}[htb]
\begin{center}
\begin{tikzpicture}[scale=.7,style=thick,x=1cm,y=1cm]
\def\vr{2.75pt}
\path (0,0) coordinate (v2);
\path (0,1) coordinate (v3);
\draw (v2) -- (v3);
\path (2,0) coordinate (u2);
\path (2,1) coordinate (u3);
\draw (u2) -- (u3);
\path (4,0) coordinate (w2);
\path (4,1) coordinate (w3);
\draw (w2) -- (w3);
\draw (v3) -- (u3);
\draw (u3) -- (w3);
\draw (v2) [fill=black] circle (\vr);
\draw (v3) [fill=black] circle (\vr);
\draw (u2) [fill=black] circle (\vr);
\draw (u3) [fill=black] circle (\vr);
\draw (w2) [fill=white] circle (\vr);
\draw (w3) [fill=black] circle (\vr);
\draw (v3) to[out=60,in=120, distance=1cm ] (w3);
\end{tikzpicture}
\caption{The corona $K_3 \circ K_1$ of $K_3$}
\label{corona}
\end{center}
\end{figure}

\medskip
By Claim~\ref{c:claimH22b}, we may assume that we may assume that all edges between $R_1'$ and $\{x',z'\}$ are present, except for at most one edge.

\begin{claim}
\label{c:claimH22c}
If all edges between $R_1'$ and $\{x',z'\}$ are present, except for exactly one edge, then the $G$-SC chain is given by part~(g) or~(h) in the statement of Lemma~\ref{lemcase2}.
\end{claim}
\begin{proof}
Suppose that all edges between $R_1'$ and $\{x',z'\}$ are present, except for exactly one edge.  By symmetry, and renaming $x'$ and $z'$ if necessary, we may assume that $x'$  is not adjacent to exactly one vertex of $R'$.

Suppose that $|L_1'| \ge 2$. In this case, the graph $B$ is a SP-graph, and the singleton coalition graph of $B$ is the graph $M_3$. As an illustration, if $B$ is the graph illustrated on the left hand side of Figure~\ref{CGbig}, then the singleton coalition graph of $B$ is the graph $\CG(B, \Gamma_1) = M_3$ illustrated on the right hand side of Figure~\ref{CGbig}. However in this case, there is no vertex $\tilde{u}$ in the graph $\CG(B, \Gamma_1)$ such that $\{\tilde{z_1}\}$ forms a coalition with $\{\tilde{u}\}$, implying that $\CG(B,\Gamma_1)$ is not a SP-graph, and therefore the $G$-SC chain is the sequence $G \rightarrow B \rightarrow M_3$ given by part~(g) in the statement of Lemma~\ref{lemcase2}.

\begin{figure}[htb]
\begin{center}
	\includegraphics[width=0.7\linewidth]{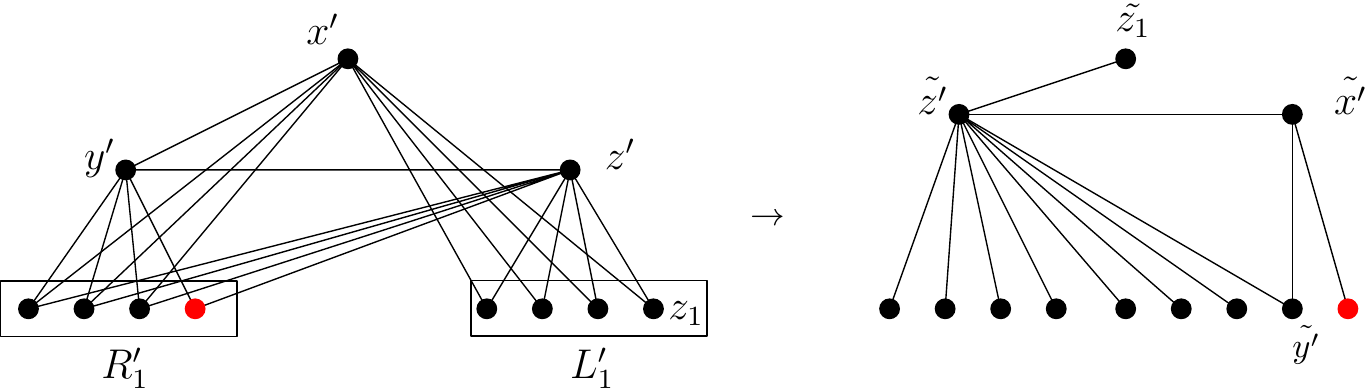}
	\caption{Left:  the graph $B$. Right: the graph  $\CG(B, \Gamma_1)$.}
  \label{CGbig}
	\end{center}
\end{figure}

Hence, we may assume that $|L_1'| = 1$. Let $L_1' = \{z_1\}$. Since $|L_1'| + |R_1'| \ge 3$, we note that $|R_1'| \ge 2$. In this case, the set $\{z_1\}$ forms a coalition with each of the sets $\{y'\}$ and $\{z'\}$, while all other coalitions from the previous case when $|L_1'| \ge 2$ remain unchanged. The graph $B$ is once again a SP-graph, and the singleton coalition graph of $B$ is the graph $M_4$. In the example in Figure~\ref{CGbig} in the case when $L_1' = \{z_1\}$, the graph $M_4$ is obtained from the graph $M_3$ on the right hand side of Figure~\ref{CGbig} by adding the edge between the vertices $\tilde{y'}$ and $\tilde{z_1}$. However noting that $|R_1'| \ge 2$, we once again infer that there is no vertex $\tilde{u}$ in the graph $\CG(B, \Gamma_1)$ such that $\{\tilde{z_1}\}$ forms a coalition with $\{\tilde{u}\}$, implying that $\CG(B,\Gamma_1)$ is not a SP-graph, and therefore the $G$-SC chain is the sequence $G \rightarrow B \rightarrow M_4$ given by part~(h) in the statement of Lemma~\ref{lemcase2}.~\smallqed
\end{proof}

\medskip
We now return to the proof of Lemma~\ref{lemcase2} one final time. By Claim~\ref{c:claimH22c}, we may assume that we may assume that all edges between $R_1'$ and $\{x',z'\}$ are present.

Suppose that $|L_1'| \ge 2$. In this case, the set $\{y'\}$ forms a coalition with only the sets $\{x'\}$ and $\{z'\}$. Each of the sets $\{x'\}$ and $\{z'\}$ form a coalition with every other singleton set in the partition $\Gamma_1$. Moreover, each vertex in $R_1' \cup L_1'$ forms a coalition with only the sets $\{x'\}$ and $\{z'\}$.  Thus in this case, the graph $B$ is a SP-graph, and the singleton coalition graph of $B$ is the graph $\CG(B, \Gamma_1) = K_2 + \overline{K}_{n-2}$.  However, the graph $K_2 + \overline{K}_{n-2}$ contains two full vertices and is not a singleton coalition graph noting that $n \ge 6$. As an illustration, if $B$ is the graph illustrated on the left hand side of Figure~\ref{CGbig2}, then the singleton coalition graph of $B$ is the graph $\CG(B, \Gamma_1) = K_2 + \overline{K}_{n-2}$ illustrated on the right hand side of Figure~\ref{CGbig2}. Therefore, $\CG(B,\Gamma_1)$ is not a SP-graph, and the $G$-SC chain is the sequence $G \rightarrow B \rightarrow K_2 + \overline{K}_{n-2}$ given by part~(i) in the statement of Lemma~\ref{lemcase2}.

\begin{figure}[htb]
\begin{center}
	\includegraphics[width=0.7\linewidth]{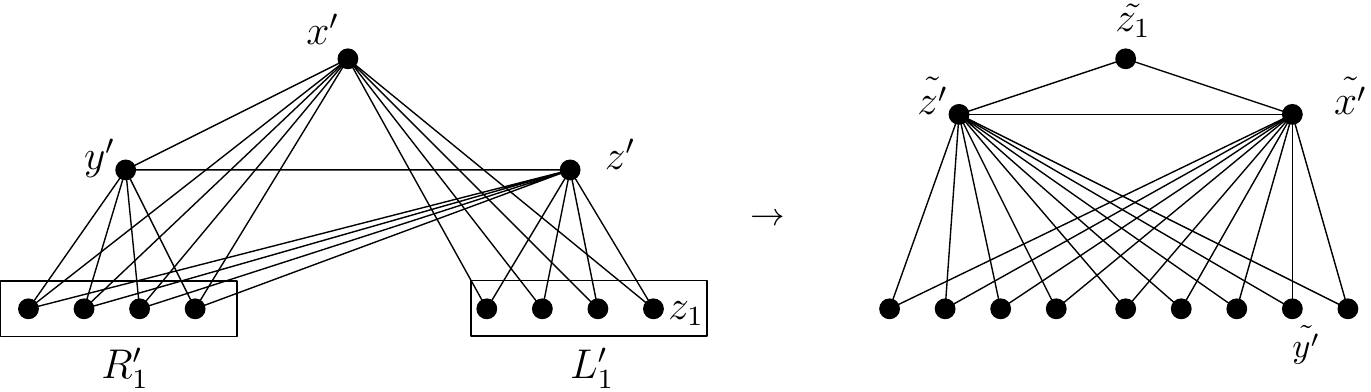}
	\caption{Left:  the graph $B$. Right: the graph  $\CG(B, \Gamma_1)$.}
  \label{CGbig2}
	\end{center}
\end{figure}

Hence, we may assume that $|R_1'| \ge 2$ and $|L_1'| = 1$. Let $L_1' = \{z_1\}$. In this case, the set $\{z_1\}$ forms a coalition with each of the sets $\{x'\}$, $\{y'\}$ and $\{z'\}$, while all other coalitions from the previous case when $|L_1'| \ge 2$ remain unchanged. Thus in this case, the graph $B$ is a SP-graph, and the singleton coalition graph of $B$ is the graph $\CG(B, \Gamma_1) = (K_2 + \overline{K}_{n-2}) + e$ obtained from $K_2 + \overline{K}_{n-2}$ by adding an edge between two vertices in the independent set. (In the example of Figure~\ref{CGbig2}  in the case when $L_1' = \{z_1\}$, the edge added to the graph on the right hand side of Figure~\ref{CGbig2} is the edge joining $\{\tilde{z_1}\}$ and $\{\tilde{y'}\}$.) However, the graph $(K_2 + \overline{K}_{n-2}) + e$ contains two full vertices and is not a singleton coalition graph noting that $n \ge 6$. Therefore, $\CG(B,\Gamma_1)$ is not a SP-graph, and the $G$-SC chain is the sequence $G \rightarrow B \rightarrow (K_2 + \overline{K}_{n-2}) + e$ given by part~(j) in the statement of Lemma~\ref{lemcase2}. This completes the proof of Lemma~\ref{lemcase2}.~\QED
\end{proof}

\medskip
We consider next the case when the graph $B$ belongs to the family ${\cal H}_2^3$.

\begin{lemma}
\label{lemcase3}
If $B \in {\cal H}_2^3$, then the $G$-SC chain is one of the following chains, where $H_1$ is a graph in the family ${\cal H}_1$. \\ [-24pt]
\begin{enumerate}
\item[{\rm (a)}] $G\rightarrow B\rightarrow H_1$,
\item[{\rm (b)}] $G\rightarrow B\rightarrow C_4\rightarrow K_4\rightarrow \overline{K}_4$,
\item[{\rm (c)}] $G\rightarrow B\rightarrow K_{2,n-2} \rightarrow K_1+K_{1,n-2}$,
\item[{\rm (d)}] $G\rightarrow C_5\rightarrow C_5\rightarrow \cdots$,
\item[{\rm (e)}] $G \rightarrow M_1 \rightarrow \left(\, \overline{K}_2+(K_1\cup K_2)\right)\rightarrow  \overline{K}_2+K_3\rightarrow \overline{K}_3\cup K_2$,
\item[{\rm (f)}] $G\rightarrow \overline{K}_3+\overline{K}_2 \rightarrow \overline{K}_2+K_3\rightarrow \overline{K}_3\cup K_2$,
\item[{\rm (g)}] $G\rightarrow\left(\, \overline{K}_2+(K_1\cup K_2)\right)\rightarrow \overline{K}_2+K_3\rightarrow \overline{K}_3\cup K_2$,
\item[{\rm (h)}] $G\rightarrow B\rightarrow \CG(B,\Gamma_1)\rightarrow \overline{K}_4$,
\item[{\rm (i)}] $G\rightarrow B\rightarrow \CG(B,\Gamma_1)\rightarrow\overline{K}_3\cup K_2$,
\item[{\rm (j)}] $G\rightarrow B \rightarrow \CG(B,\Gamma_1) \rightarrow \overline{K}_2\cup K_2$,
\item[{\rm (k)}]$G\rightarrow B \rightarrow \CG(B,\Gamma_1) \rightarrow\overline{K}_2\cup P_3$,
\item[{\rm (l)}] $G\rightarrow B \rightarrow  \CG(B,\Gamma_1) \rightarrow H_1$,
\item[{\rm (m)}] $G\rightarrow B \rightarrow \CG(B,\Gamma_1) \rightarrow C_4\rightarrow K_4 \rightarrow \overline{K}_4$.
\item[{\rm (n)}] $G\rightarrow B \rightarrow \CG(B,\Gamma_1) \rightarrow K_{2,n-2} \rightarrow K_1 + K_{1,n-2}$,
\item[{\rm (o)}] $G \rightarrow B \rightarrow M_1 \rightarrow \left(\, \overline{K}_2+(K_1\cup K_2)\right)\rightarrow  \overline{K}_2+K_3\rightarrow \overline{K}_3\cup K_2$,
\item[{\rm (p)}] $G\rightarrow B \rightarrow M_2\rightarrow  \overline{K}_2+K_3\rightarrow \overline{K}_3\cup K_2$,
\item[{\rm (q)}] $G \rightarrow B \rightarrow \CG(B,\Gamma_1) \rightarrow K_3 \circ K_1$,
\item[{\rm (r)}] $G\rightarrow B \rightarrow \CG(B,\Gamma_1)  \rightarrow M_3$,
\item[{\rm (s)}] $G\rightarrow B \rightarrow \CG(B,\Gamma_1)  \rightarrow M_4$,
\item[{\rm (t)}] $G\rightarrow B \rightarrow \CG(B,\Gamma_1)  \rightarrow K_2 + \overline{K}_{n-2}$,
\item[{\rm (u)}] $G \rightarrow B \rightarrow \CG(B,\Gamma_1)  \rightarrow (K_2 + \overline{K}_{n-2}) + e$,
\item[{\rm (v)}] $G\rightarrow K_{3,n-3}\rightarrow K_{3,n-3}\rightarrow \cdots$, or
\item[{\rm (w)}] $G\rightarrow (K_1\cup K_2)+\overline{K}_{n-3}\rightarrow P_3+\overline{K}_{n-3}\rightarrow K_1\cup K_{2,n-3}$.
\end{enumerate}
\end{lemma}
\begin{proof}
Suppose that the singleton coalition graph $B = \CG(G,\Gamma_1) \in {\cal H}_2^3$ (as illustrated in Figure~\ref{h2familyc31}). Recall that by assumption, the graph $G$ is a \SP with $\delta(G)=2$ and with no full vertex, and so by Theorem~\ref{thmf2} we have $G \in {\cal F}_2$. Since $\CG(G,\Gamma_1) = B \in {\cal H}_2^3$, we know from the proof of Theorem~\ref{lemh2f2} that $G\in {\cal F}_2^3$ (as illustrated in Figure~\ref{figfamily}). Adopting our notation to define the families ${\cal F}_2^3$ and ${\cal H}_2^3$, let $X = R_1\cup R_2\cup L_1 \subset V(G)$ and let $X' = R_1' \cup L_1' \cup R_2' \subset V(B)$.

We show firstly that $X' = \emptyset$. Suppose, to the contrary, that $|X'| \ge 1$. Suppose that $|X'| = 1$. By construction of the graph $G \in {\cal F}_2^3$, each vertex of $W$ is adjacent to all vertices of $X \cup L_2$. Further, $G[W]$ is a clique. Since $|X'|=1$, we have $|(X \cup L_2) \setminus W| = 1$.  Let $(X \cup L_2) \setminus W = \{g\}$, and so $\tilde{g}\notin W'$. By our earlier properties of vertices in the set $W$, the set $\{g\} \cup W$ forms a clique. The sets $\{g\}$ and $\{x\}$  form a coalition in $G$. Therefore, by the construction of $B$, the vertex  $\tilde{g} \in W'$, a contradiction. Hence,  $|R_1'\cup L_1'\cup R_2'|\ge 2$. We may assume, without loss of generality, that $R_1' \ne \emptyset$. Let $q_1, q_2 \in X'$. By construction of $B \in {\cal H}_2^3$ in Definition~\ref{def-h2}, we have $W'\ne \emptyset$ and $x'$ is adjacent only to all vertices of $W'$. Thus since $x'$ is not adjacent to any of the vertices  $q_1, q_2, y'$ and $z'$,  the set $\{q_1\}$ is not a coalition partner of $\{u\}$, where $u$ is a vertex of $B \setminus \{q_1\}$. Hence, $B$ is not a SP-graph, which is a contradiction.

Therefore, $X' = \emptyset$, implying that the vertex set of the graph $B \in {\cal H}_2^3$ is $V(B) = \{x',y',z'\} \cup W'$. Recall that $W'$ is an independent set, and that each vertex in $W'$ is adjacent to $x'$. Therefore if $w' \in W'$, then in the graph $B$ we have $N(w') \subset \{x',y'z'\}$, and so $\delta(B) \le 3$.

Suppose that $\delta(B)=0$. Since $x'$ is adjacent to all vertices of $W'$, in this case we must have $\deg_B(z')=0$ or $\deg_B(y')=0$. By assumption, $B$ is a SP-graph, and so by Theorem~\ref{thmdelta0}, we have $B \cong K_1 \cup K_{n-1}$. However since $x'$ is adjacent to neither~$y'$ nor~$z'$, such a construction of $B$ in this case is not possible. Hence, $1 \le \delta(B) \le 3$.

Suppose that $\delta(B)=1$. If $|W'| = 1$,  then  by the construction of $G\in {\cal F}_2^3$, $|W|=1$, and  since $x'$ is not adjacent to both $y'$ and $z'$ in $B$,  $\delta(G)=0$, which is a contradiction.  Hence, $|W'| \ge 2$, implying by construction that there is no full vertex in $B$. Hence applying Theorem \ref{thbyr} to the SP-graph $B$, the $B$-SC chain is one of the chains given in part~(a),~(b) or~(c) in the statement of Theorem~\ref{thbyr}. Thus, the $G$-SC chain is given by Part~(a),~(b) or~(c) in the statement of Lemma~\ref{lemcase3}. Hence, we may assume that $\delta(B) \ge 2$.

Suppose that $\delta(B) = 2$.  Since $x'$ is only adjacent to all vertices of $W'$ and $\delta(B)=2$, we have $|W'| \ge 2$. Suppose that  $|W'|=2$, and so $B$ has order~$5$. All  possible \SPs of order~$5$ are $C_5$, $M_1$,  $\overline{K}_3+\overline{K}_2$ and $\left(\overline{K}_2+(K_1\cup K_2)\right)$. Thus, in this case the $G$-SC chain is given by Part~(d),~(e),~(f) or~(g) in the statement of Lemma~\ref{lemcase3}.

Hence, we may assume that in this case $\delta(B) = 2$, we have $|W'| \ge 3$. Since $B$ is an isolate-free SP-graph, by Theorem~\ref{thmf2} we have $B \in {\cal F}_2$. Thus, $B \in {\cal F}_2^1 \cup {\cal F}_2^2 \cup {\cal F}_2^3$. We show that $B \notin {\cal F}_2^3$. Suppose, to the contrary, that $B \in{{\cal F}_2^3}$. Consider the vertices $x,y$ and $z$, and the sets $W, R_1, L_1, R_2$ and $L_2$ used in the definition of ${\cal F}_2^3$ (see Definition \ref{def-f2}).

Since $|W'| \ge 3$, the vertex $x'$ has degree at least~$3$ in $B$. Since the vertex~$x$ has degree~$2$ in ${\cal F}_2^3$, the vertex $x' \ne x$. If $x=z'$ and $y=y'$, then $z \in W'$. Hence, $W=\{x\}$ and $W'= X \cup L_2$. Since $W'$ is an independent set of $B$,  there are no edges between the vertices of $X \cup L_2$, which contradicts the definition of the family  ${\cal F}_2^3$. If $x=z'$,  $y\in W'$ and  $z\in W'$, then $y',z' \notin W$.  Hence, $W\subseteq W'$. If $|W\cap W'|=1$, then it is not hard to see that $R_2=\emptyset$ or $L_1=\emptyset$ (see Definition \ref{def-f2}), which is a contradiction. If  $|W\cap W'|>1$,  then since $W'$  is an independent set, $G[W]$ is not a clique, which is a contradiction.  The other cases for the selections of $x,y,$, and $z$ analogously leads to a contradiction. Hence, $B \not\in {{\cal F}_2^3}$. Thus, $B\in {\cal F}_2^1 \cup {\cal F}_2^2$. Hence by Theorem~\ref{lemh2f2}, $\CG(B,\Gamma_1) \in {\cal H}_2$ and by Definition~\ref{def-h2}, we infer that $\CG(B,\Gamma_1) \in {\cal H}_2^1$ or $\CG(B,\Gamma_1) \in {\cal H}_2^2$. Hence, using Lemmas~\ref{lemcase1} and~\ref{lemcase2}, the $G$-SC chain is given by Part~(l), (m), (n), (o), (p), (q), (r), (s), (t), or~(u) in the statement of Lemma~\ref{lemcase3}.

Hence, we may assume that $\delta(B) = 3$. If $y'$ is not adjacent to $z'$, then $B\cong K_{3,n-3}$. In this case, the $G$-SC chain is given by Part~(v) in the statement of Lemma~\ref{lemcase3}. If $y'$ is adjacent to $z'$, then $B\cong (K_1\cup K_2)+\overline{K}_{n-3}$. In this case, the $G$-SC chain is given by Part~(w) in the statement of Lemma~\ref{lemcase3}. This completes the proof of Lemma~\ref{lemcase3}.~\QED
\end{proof}

\medskip
As an immediate consequence of Lemmas~\ref{lemcase1},~\ref{lemcase2}, and~\ref{lemcase3}, we have the following result.

\begin{theorem}
\label{thmlsccfinal}
If $G$ is a \SP of order $n$ with $\delta(G)=2$ that contains no full vertex, then $L_{\SCC}(G)=\infty$ or $L_{\SCC}(G)\le 5$.
\end{theorem}

\section{Conclusion}

In this paper, we have addressed two open problems posed by Haynes et al (see \cite{coal0,coal3}). The first problem was to characterize all graphs $G$ of order $n$ and minimum degree $\delta(G)=2$, such that $C(G)=n$.  The second problem was to investigate singleton coalition graph chains. We showed that there exist singleton coalition graph chains with the infinite length. Moreover, we characterized the singleton coalition graph chains starting with the graphs $G$ satisfying $\delta(G) \le 2$. It would be interesting to characterize the singleton coalition graph chains starting with the graphs $G$ satisfying $\delta(G) \ge 3$.

\end{document}